\documentclass[11pt, a4paper]{amsart}

\usepackage[T1]{fontenc}

\usepackage{geometry}
\usepackage{amssymb,amsmath,mathtools,mathrsfs,bm,bbm,esint}

\usepackage[usenames,dvipsnames,svgnames,table]{xcolor}

\allowdisplaybreaks
\usepackage{combelow}
\usepackage{appendix}

\usepackage{enumitem}
\setenumerate{label={\rm (\alph{*})}}

\usepackage{algpseudocode}
\usepackage{algorithm}

\usepackage[colorlinks=true,citecolor=black,urlcolor=black,
linkcolor=black]{hyperref}

\usepackage{graphicx}
\usepackage{tikz}

\makeatletter
\newcommand*\bigcdot{\mathpalette\bigcdot@{.5}}
\newcommand*\bigcdot@[2]{\mathbin{\vcenter{\hbox{\scalebox{#2}{$\m@th#1\bullet$}}}}}
\makeatother

\makeatletter
\DeclareFontFamily{OMX}{MnSymbolE}{}
\DeclareSymbolFont{MnLargeSymbols}{OMX}{MnSymbolE}{m}{n}
\SetSymbolFont{MnLargeSymbols}{bold}{OMX}{MnSymbolE}{b}{n}
\DeclareFontShape{OMX}{MnSymbolE}{m}{n}{
    <-6>  MnSymbolE5
   <6-7>  MnSymbolE6
   <7-8>  MnSymbolE7
   <8-9>  MnSymbolE8
   <9-10> MnSymbolE9
  <10-12> MnSymbolE10
  <12->   MnSymbolE12
}{}
\DeclareFontShape{OMX}{MnSymbolE}{b}{n}{
    <-6>  MnSymbolE-Bold5
   <6-7>  MnSymbolE-Bold6
   <7-8>  MnSymbolE-Bold7
   <8-9>  MnSymbolE-Bold8
   <9-10> MnSymbolE-Bold9
  <10-12> MnSymbolE-Bold10
  <12->   MnSymbolE-Bold12
}{}

\let\llangle\@undefined
\let\rrangle\@undefined
\DeclareMathDelimiter{\llangle}{\mathopen}%
                     {MnLargeSymbols}{'164}{MnLargeSymbols}{'164}
\DeclareMathDelimiter{\rrangle}{\mathclose}%
                     {MnLargeSymbols}{'171}{MnLargeSymbols}{'171}
\makeatother

\numberwithin{equation}{section}

\newtheorem{theorem}{Theorem}[section]
\newtheorem{lemma}[theorem]{Lemma}
\newtheorem{proposition}[theorem]{Proposition}

\newtheorem{definition}[theorem]{Definition}

\theoremstyle{remark}
\newtheorem{remark}[theorem]{Remark}

\usepackage{booktabs}

\normalsize

\def\XXint#1#2#3{{\setbox0=\hbox{$#1{#2#3}{\int}$ }
		\vcenter{\hbox{$#2#3$ }}\kern-.6\wd0}}

\newcommand{\mres}{\mathbin{\vrule height 1.6ex depth 0pt width
		0.13ex\vrule height 0.13ex depth 0pt width 1.3ex}}

\newcommand{\bv}{\operatorname{BV}}

\newcommand{\dif}{\mathrm{d}}
\newcommand{\D}{\mathrm{D}}

\newcommand{\R}{\mathbb{R}}
\newcommand{\B}{\mathscr{B}}
\newcommand{\C}{\mathbb{C}}

\newcommand{\wstar}{\stackrel{*}{\rightharpoonup}}

\newcommand{\ccinfty}{\operatorname{C}_{c}^{\infty}}
\newcommand{\sobo}{\operatorname{W}}

\renewcommand{\geq}{\geqslant}

\newcommand{\hold}{\operatorname{C}}

\renewcommand{\leq}{\leqslant}

\newcommand{\sym}{\operatorname{sym}}

\newcommand{\lin}{\operatorname{Lin}}

\usepackage{tikz-cd}
\newcommand{\U}{\mathbb U}
\newcommand{\V}{\mathbb V}

\newcommand{\eps}{\bm{\varepsilon}}
\newcommand{\ym}{\bm{\nu}}
\newcommand{\Y}{\mathrm{Y}}

\newcommand{\wk}{\rightharpoonup}
\newcommand{\Adm}{\mathrm{Adm}}
\newcommand{\gt}{\mathrm{gt}}
\newcommand{\PSNR}{\mathrm{PSNR}}
\newcommand{\stat}{\mathrm{stat}}

\AtEndDocument{\bigskip\footnotesize

	\noindent V.~Pagliari.
	
	\noindent Institute of Analysis and Scientific Computing,
	TU Wien,
	
	\noindent Wiedner Hauptstra\ss e 8-10, 1040 Vienna, Austria.
	
	\noindent E-mail: \href{mailto:valerio.pagliari@tuwien.ac.at}{\tt valerio.pagliari@tuwien.ac.at}.
	
	\medskip
	\noindent K.~Papafitsoros.
	
	\noindent Weierstrass Institute for Applied Analysis and Stochastics,
	
	\noindent Mohrenstra\ss e 39, 10117 Berlin, Germany.
	
	\noindent E-mail: \href{mailto:papafitsoros@wias-berlin.de}{\tt papafitsoros@wias-berlin.de}.
	
	\medskip
	\noindent B.~Rai\cb{t}\u{a}.
	
	\noindent Centro di Ricerca Matematica Ennio De Giorgi, Scuola Normale Superiore,
	
	\noindent P.za dei Cavalieri, 3, 56126 Pisa PI, Italy.
	
	\noindent E-mail: \href{mailto:raita@mis.mpg.de}{\tt raita@mis.mpg.de}
	
	\medskip
	\noindent A.~Vikelis.
	
	\noindent University of Sussex,
	
	\noindent Falmer, Brighton BN1 9RH, United Kingdom.
	
	\noindent E-mail: \href{mailto:A.Vikelis@sussex.ac.uk}{\tt A.Vikelis@sussex.ac.uk}
}

\title[Bilevel schemes for TV-type functionals in imaging]{Bilevel~training~schemes~in~imaging \\ for total-variation-type functionals \\ with~convex~integrands}
\author[V. Pagliari]{Valerio Pagliari}	
\author[K. Papafitsoros]{Kostas Papafitsoros}	
\author[B. Rai\cb{t}\u{a}]{Bogdan Rai\cb{t}\u{a}}
\author[A. Vikelis]{Andreas Vikelis}

\begin{document}
\begin{abstract}
In the context of image processing,
given a $k$-th order, homogeneous and linear differential operator with constant coefficients,
we study a class of variational problems 
whose regularizing terms depend on the operator.
Precisely, the regularizers are integrals
of spatially inhomogeneous integrands
with convex dependence on the differential operator applied to the image function.
The setting is made rigorous by means of the theory of Radon measures
and of suitable function spaces modeled on $\bv$. 
We prove the lower semicontinuity of the functionals at stake
and existence of minimizers for the corresponding variational problems.
Then, we embed the latter into a bilevel scheme
in order to automatically compute the space-dependent regularization parameters,
thus allowing for good flexibility and preservation of details in the reconstructed image. 
We establish existence of optima for the scheme
and we finally substantiate its feasibility
by numerical examples in image denoising.
The cases that we treat are Huber versions of the first and second order total variation
with both the Huber and the regularization parameter being spatially dependent. Notably the spatially dependent version of second order total variation produces high quality reconstructions when compared to regularizations of similar type, and the introduction of the spatially dependent Huber parameter leads to a further enhancement of the image details.
\end{abstract}
\maketitle


 \section{Introduction}
 
 In this contribution we study a bilevel training scheme
for the automatic selection of spatially varying regularization weights
in the framework of variational image reconstruction.
Specifically, given a suitably defined class $\Adm$ of admissible weights $\alpha$,
we look for solutions to the problem
\begin{equation}\label{eq:level1}
	\alpha^\ast \in {\rm argmin} \left\{ F(u_\alpha) : \alpha \in \Adm \right\},
\end{equation}
where $F$ is an assigned cost functional and
$u_\alpha$ is an image reconstructed by minimizing
\begin{equation}\label{eq:I=Phi+R}
	I[u;\alpha] \coloneqq \Phi_g(u) + \mathcal{R}(u;\alpha).
\end{equation}
Here, $\Phi_g$ is a fidelity term
that penalizes deviations of $u$ from the datum $g$, whereas
$\mathcal{R}(u;\alpha)$ is a regularization functional
whose strength can be tuned by an appropriate selection of the regularization parameter $\alpha$ belonging to the admissible set $\Adm$.
The datum $g$ is typically a corrupted version of some ground truth image $u_{\gt}$.
Often, one has 
\begin{equation*}
	g = T u_{\gt} +\eta,
\end{equation*}
with  $\eta$ denoting a random noise component and
 $T$ being a bounded linear operator
 that corresponds to a certain image reconstruction problem.
 For instance, $T$ is a blurring operator in the case of deblurring,
 a sub-sampled Fourier transform in magnetic resonance imaging (MRI), 
 the Radon transform in tomography,
 or simply the identity in denoising tasks, on which we will be focusing here.
 The aim of solving a problem of the type \eqref{eq:I=Phi+R} for suitable $\Phi_g$, $\mathcal R$ and $\alpha$
 is to obtain an output $u$
 which represents as well as possible the initial ground truth image $u_{\gt}$.
 We concisely point out here that
 the main novelty of the present paper consists in establishing existence of solutions to the scheme
for inhomogeneous regularizers of the type
\begin{equation}\label{first_intro:BV-B}
	\mathcal R(u;\alpha) = \int_{\Omega}\alpha(x) f(x,\dif \B u),
\end{equation}
and present numerical results that fit this framework.
Here, $\Omega\subset \mathbb{R}^{n}$ is the image domain,
$\alpha\in \Adm\subset L^\infty(\Omega,[0,+\infty))$ belongs to a class of admissible weights,
$f$ is a Carath\'eodory integrand that is convex in the second entry, and
$\B$ is a linear, $k$-th order, homogeneous differential operator with constant coefficients. Before we report further on our contribution,
we proceed with a brief review of regularization functionals in image reconstruction.

%

Among classical regularization functionals we find the total variation (TV) \cite{ROF92, ChambolleLions},
as well as higher order or anisotropic extensions of it.
Particularly relevant for this work are the second order total variation (TV$^{2}$) \cite{PaSc, Piffet} and  the total generalized variation (TGV) \cite{TGV}.
For a function $u\in L^{1}(\Omega)$,
these functionals are defined by duality as follows: 
\begin{align}
\mathrm{TV}(u)&= \sup \left \{\int_{\Omega} u\,\mathrm{div}\phi\, \dif x: \; \phi\in\ccinfty(\Omega, \mathbb{R}^{n}), \; \|\phi\|_{\infty}\leq 1 \right \}\label{def:TV},\\
\mathrm{TV}^{2}(u)&= \sup \left \{\int_{\Omega} u\,\mathrm{div}^{2}\phi \, \dif x: \; \phi\in \ccinfty(\Omega, \mathbb{R}^{n\times n}), \; \|\phi\|_{\infty}\leq 1 \right \}\label{def:TV2},\\
\mathrm{TGV}(u)&=  \sup \left \{\int_{\Omega} u\,\mathrm{div}^{2}\phi \, \dif x: \; \phi\in \ccinfty(\Omega, {\mathcal{S}^{n\times n}}), \; \|\phi\|_{\infty}\leq \alpha_{0}, \; \|\mathrm{div}^{2}\phi\|_{\infty}\leq \alpha_{1} \right \}.\label{def:TGV}
\end{align}
{Here $\mathcal{S}^{n\times n}$ denotes the space of $n\times n$ symmetric matrices}.
Note that the scalar regularization parameters $\alpha_{0}, \alpha_{1}>0$ are inserted within the definition of TGV,
while the other functionals admit a single weighting parameter $\alpha$ acting in a multiplicative way,
i.e., $\alpha \mathrm{TV}$ and $\alpha\mathrm{TV}^{2}$.
If the supremum in \eqref{def:TV} is finite,
then we say that $u\in \bv(\Omega)$, the space of functions of bounded variation \cite{AFP00},
and $\mathrm{TV}(u)=|\D u|(\Omega)$, where $|\D u|$ is the total variation measure associated with the distributional derivative $\D u\in \mathcal{M}(\Omega, \mathbb{R}^{n})$.
Similarly, if the right-hand side in \eqref{def:TV2} is finite,
then $u\in \bv^{2}(\Omega)$, the space of functions of bounded second variation \cite{PaSc, Piffet},
and $\mathrm{TV}^{2}(u)=|\D^{2}u|(\Omega)$, with $\D^{2}u\in \mathcal{M}(\Omega, \mathcal{S}^{n\times n})$.
Finally, it turns out that if the supremum in \eqref{def:TGV} is finite,
then $u\in \mathrm{BV}(\Omega)$ as well and
    $$
    \mathrm{TGV}(u)=\min_{w\in \mathrm{BD}(\Omega)}
    \left\{
    \alpha_{1}\int_{\Omega} \dif|\D u-w| + \alpha_{0}\int_{\Omega} \dif|\mathscr{E}w|
    \right\},
    $$
see \cite{BredValk, bredies2014regularization}.
In the previous formula, $\mathrm{BD}(\Omega)$ is the space of functions of bounded deformation and
$\mathscr{E}w$ denotes the symmetrized gradient of $w$.
The advantage of higher order regularizers lies in their capability
to reduce an undesirable artifact typical of TV, the so-called staircasing effect, that is,
the creation of cartoon-like piecewise constant structures in the reconstruction \cite{papafitsorosphd}.

Grounding on the concept of convex functions of Radon measures \cite{Dem},
variants of the above regularizers involving convex integrands have also been considered in the literature \cite{PaSc, vese2001study, stadler}.
A widely used example is the one of Huber total variation $\mathrm{TV}_{\gamma}$,
which is defined for $u\in \mathrm{BV}(\Omega)$ as 
\begin{equation}\label{intro:hTV}
\mathrm{TV}_{\gamma}(u)=\int_{\Omega} f_{\gamma}(\dif \D u)= \int_{\Omega} f_{\gamma}(\nabla u) \dif x+ \int_{\Omega} \dif|\D^{s}u|,
\end{equation}
with $\nabla u$ and $\D^{s}u$ denoting respectively
the absolutely continuous and the singular part of $\D u$ with respect to the Lebesgue measure.
The function $f_{\gamma}\colon \mathbb{R}^{n} \to \mathbb{R}$ is given for $\gamma\ge 0$ by
\begin{equation}\label{intro:f_gamma}
f_{\gamma}(z)=
\begin{cases}
|z|-\dfrac{1}{2} \gamma, & \text { if } |z|\ge \gamma,\\[0.8em]
\dfrac{1}{2\gamma} |z|^{2}, & \text { if } |z| < \gamma.
\end{cases}
\end{equation}
Note that $\mathrm{TV}_{\gamma}(u)$ can be equivalently defined via duality as
\begin{equation}\label{intro:hTV_dual}
 \mathrm{TV}_{\gamma}(u)=\sup \left \{\int_{\Omega} u\,\mathrm{div}\phi \, \dif x  
 -\frac{1}{2} \gamma\int_{\Omega}  |\phi|^{2} \dif x : \phi\in \ccinfty(\Omega, \mathbb{R}^n), \; \|\phi\|_{\infty}\leq 1  \right \},
 \end{equation}
see \cite{Dem}.
This modification of TV,
which corresponds to a smoothing of the $\ell_{1}$ norm in the discrete setting,
is typically considered
in order to employ classical smooth numerical solvers for the solution of the minimization problem \eqref{eq:I=Phi+R}.
In this specific case, however, it also leads to a reduction of the staircasing effect
by penalizing small gradients with the Tikhonov term $(2\gamma)^{-1} \int_{\Omega} |\nabla u|^{2}dx$,
which promotes smooth reconstructions \cite{stadler, journal_tvlp}.

In all these models, the choice of the weights in the regularization term is crucial
to establish an adequate balance between data fitting and denoising. On one hand, the reconstructed $u$ may remain too noisy or have many artifacts  if the regularization is too weak.
On the other hand, a very strong regularization may result in an unnatural smoothing effect.
In the last years, bilevel minimization methods have been employed to select these weights automatically.
A subfamily of these methods assumes the existence of one or several training pairs $(u_{\gt}, g)$
consisting of the ground truth and its corrupted counterpart
\cite{bilevellearning, TGV_learning2, GVlearningfirst, noiselearning, pockbilevel}.
In these works the energy in the upper level problem \eqref{eq:level1} is usually given by 
\begin{equation}\label{intro:ul_psnr}
F_{\PSNR}(u)\coloneqq\|u-u_{\gt}\|_{L^{2}(\Omega)}^{2},
\end{equation}
 and its minimization essentially corresponds to finding reconstructions that are closest to the ground truth in the least square sense.
Since typical images generally feature both homogeneous regions and fine details,
it is reasonable to assume that the optimal  regularization intensity is not uniform throughout the domain. This matter of fact has prompted researchers to consider bilevel schemes that output space-dependent weights, i.e., functions $\alpha\colon \Omega \to [0,+\infty)$ \cite{carola_spatial, reyes2021optimality}.
A recent series of papers \cite{HiRa17, hintermuellerPartII, bilevel_handbook, bilevelTGV} deals with schemes for TV and TGV
that yield such weights without resorting to the ground truth.
If the corrupted datum $g$ is obtained by some additive Gaussian noise $\eta$ with variance $\sigma^{2}$,
this is achieved by the introduction of the statistics-based upper level objective 
\begin{align}\label{intro:ul_stats}
  F_{\stat}(u) \coloneqq \frac{1}{2} \int_{\Omega} \max(Ru-\overline{\sigma}^{2},0)^{2} \dif x + \frac{1}{2} \int_{\Omega} \min (Ru-\underline{\sigma}^{2},0)^{2} \dif x,
\end{align}
where $\underline{\sigma}^{2}\coloneqq \sigma^{2}-\epsilon$, $\overline{\sigma}^{2}\coloneqq \sigma^{2}+\epsilon$, and
\begin{align*}
  {R}u(x)\coloneqq \int_{\Omega} w(x,y) (u-g)^{2}(y)\dif y\;\; \text{ for } w\in L^{\infty}(\Omega\times \Omega), \text{  } \int_{\Omega}\int_{\Omega} w(x,y) \dif x \dif y=1.
\end{align*}
 The idea is that if the reconstructed image $u$ is close to $u_{\gt}$, then it is expected that  on average the value of $R u(x)$ will be close to $\sigma^{2}$. This justifies the use of $F_{\stat}$, since its minimization  forces the localized residuals $R u$ to fall within the tight corridor $[\underline{\sigma}^{2}, \overline{\sigma}^{2}]$.

\subsection*{Our contribution}
The contribution of this paper is connected to the aforementioned literature on several levels.
Starting from an arbitrary $k$-th order, homogeneous, linear differential operator $\B$ between two finite dimensional Euclidean spaces $\mathbb{U}$ and $\mathbb{V}$, 
we introduce the general regularizer 
\begin{equation}\label{our_reg}
\mathcal R(u;\alpha) = \sum_{i=1}^{k-1}\int_{\Omega}\alpha_i(x)  f_i(x,\dif \D^i u)+\int_{\Omega}\alpha_k(x) f_k(x,\dif \B u),
\end{equation}
$\alpha_{i}\colon \Omega \to [0,+\infty)$ being for $i=1,\ldots,k$ the spatially dependent weights.
The functions $f_{i}$ are of linear growth and convex in the second variable.
We assume them to be Carath\'eodory integrands, or in other words,
they explicitly depend on the spatial variable $x$ in a measurable way.
More details about the setting are to be found in Section \ref{sec:setup}.
As a first contribution,
we prove lower semicontinuity of the functional in \eqref{our_reg} with respect to a suitable weak-$\ast$ convergence,
a necessary step towards the existence of solutions of the corresponding variational image reconstruction problem \eqref{eq:I=Phi+R}.
Secondly, we introduce and prove existence of solutions to the bilevel scheme,
which provides an optimal spatially dependent weight and an associated reconstructed image.

Not much work has been done for functionals depending on general differential operators.
One example comes from the recent preprint \cite{DFL19},
where a bilevel scheme for first order differential operators $\B$ is developed.
Interestingly, the authors identify classes of operators $\B$
such that the scheme outputs an optimal reconstructed image and an optimal $\B$ for the upper level problem.
However, in their analysis one always obtains $\bv$ minimizers.
In contrast, in our method the operator $\B$ is fixed,
but it is allowed to be \emph{arbitrary} (see Theorems \ref{thm:lsc_any_B} and \ref{thm:main}).

From the theoretical point of view,
one of the main advantages of our approach is the fact that
we can allow for \emph{spatially dependent weights}  and for \emph{general convex integrands} in the regularizers.
Our hypotheses on the convex integrands are optimal, due to the use of Young measures for oscillation and concentration, see Appendix \ref{sec:YM}. From an analytical point of view, our regularity assumptions on the weights are minimal, as can be seen from Section \ref{sec:why-Adm}. In the future, we aim to develop our theory to include optimization problems over linear PDE operators $\B$ that satisfy as few assumptions as possible.
We expect that our lower semicontinuity result, Theorem \ref{thm:lsc_any_B}, and the existence result for the bilevel schemes, Theorem \ref{thm:main} will serve as preliminary work in this direction.


We conclude the paper with a series of numerical examples
that deal with versions of the Huber TV and TV$^{2}$ in which  both the Huber and the regularization parameter are spatially dependent. We devise a strategy to prefix the former in a sensible way, while the latter is computed automatically by the bilevel scheme. Even though the main purpose of these numerical examples is to support the applicability and versatility of the framework,
we are able to draw two interesting conclusions.
The first one is that the bilevel weighted TV$^{2}$ in combination with the statistics-based upper level objective $F_{\stat}$ is able to produce high quality reconstructions, even outperforming TGV (both in its scalar and weighted versions). The second one is that the introduction of the spatially varying Huber parameter can further enhance the detailed areas in the reconstructed images.

\subsection*{Structure of the paper}

In Section \ref{sec:setup} we introduce the spaces of functions of bounded $\B$-variation, 
which provide the functional setting for our analysis. We then state the assumptions under which the general bilevel scheme is studied, and
we justify our choice of admissible weights.
In Section \ref{sec:LSC}, we prove lower semicontinuity and existence theorems concerning the lower level of the bilevel scheme,
while Section \ref{sec:bilevel} is devoted to the main existence result for optimal weights and reconstructed images.
Eventually, numerical experiments for the weighted Huber TV and TV$^2$ regularizers are performed in Section \ref{sec:num},
where we briefly describe the algorithmic set-up and present a series of numerical examples.

\subsection*{Acknowledgements}
The authors are grateful to Elisa Davoli for preliminary discussions at the early stage of the project.
VP acknowledges the supports of the Austrian Science Fund (FWF) through project I4052 and of the BMBWF through the OeAD-WTZ num. CZ04/2019.


\section{Mathematical setup}\label{sec:setup}
We collect in this section all the assumptions and the notations to be used in the sequel.
{We also include some heuristic motivation for the definition of the class of admissible weights.}

\subsection{Functional setting: $\bv^\B_p$ spaces.}
We work in the $n$-dimensional Euclidean space $\R^n$, $n\geq 2$,
that we endow with the Lebesgue measure $\mathscr L^n$.
We let $\Omega \subset \R^n$ be a fixed open and bounded set  {with Lipschitz boundary},
which stands as the image domain.
In typical applications $n=2$ and $\Omega$ is a rectangle.
We suppose that the image functions take values in a finite dimensional inner product space $\U$,
which, for instance, is $\R$ for grayscale images, $\R^3$ for RGB images, or it can be even more structured like e.g.\ $\mathcal{S}^{n\times n}$ for diffusion tensor imaging \cite{diff_tens}.
In order to describe further the functional setting in which our analysis is carried out,
we need to introduce the class of differential operators that we consider.

Let $\V$ be another finite dimensional inner product space and let $\lin (\U,\V)$ be the space of linear maps from $\U$ to $\V$.
Hereafter, for $k\in \mathbb{N}\setminus\{0\}$,
$\B$ denotes a $k$-th order, homogeneous and linear differential operator with constant coefficients.
Explicitly,
given $B_i\in \lin(\U,\V)$ for any $n$-dimensional multi-index $i = (i_1,\dots,i_n)\in \mathbb{N}^n$,
we define for a smooth function $u\colon\R^n\rightarrow \U$
    $$
    \B u\coloneqq \sum_{|i|=k} B_i \partial^i u.
    $$
(recall that $|i| \coloneqq \sum_j i_j$).
When $u$ is less regular, we interpret $\B u$ in the distributional sense.
In particular, we are interested in the case in which $\B u$ is a finite Radon measure.

Given a generic open set $O\subset \R^n$,
we recall that
a finite ($\U$-valued) Radon measure on $O$
is a measure on the $\sigma$-algebra of the Borel sets of $O$.
We denote the space of such measures by $\mathcal M(O,\mathbb U)$, and,
by means of the classical Riesz's representation theorem,
we can identify it as the dual of the space
\[
    \hold_0(O,\U)\coloneqq \big\{ u\colon O \to \U : \{|u|>\delta\} \text{ is relatively compact for all } \delta>0 \big\},
\]
equipped with the uniform norm.
The dual norm induced on $\mathcal M(O,\mathbb U)$ turns out to be the one associated with the total variation,
which we denote by $|\,\bigcdot\,|$.
We refer to \cite[Chapter 1]{AFP00} for further reading on measure theory.

In the case we consider,
given $\mu\in \mathcal M(\Omega,\mathbb U)$,
we have that $\B \mu \in \mathcal M(\Omega,\mathbb V)$ if and only if
there exists $\nu \in \mathcal M(\Omega,\mathbb V)$ such that
\begin{equation*}
    \langle \nu, \phi \rangle = \int_\Omega \B^\ast \phi \,\dif \mu
    \qquad \text{for all } \phi\in \ccinfty(\Omega,\V),
\end{equation*}
where  $\langle \,\bigcdot\, , \,\bigcdot\, \rangle$ denotes the duality pairing and
$\B^\ast$ is the formal adjoint of $\B$, i.e.
\[
    \B^\ast \phi \coloneqq  -  \sum_{|i|=k} B_i^\ast \partial^i \phi
    \qquad \text{for all }\phi\in\ccinfty(\R^n,\V),
\]
$B_i^\ast$ being the transpose of $B_i$.

It is convenient to have at our disposal
a specific notation for the spaces
that we are going to work with.
For $\Omega$, $\U$ and $\V$ as above, and for $p\in(1,+\infty)$,
we set
    \begin{gather*}
    \bv^\B_p(\Omega) \coloneqq\{u\in L^p(\Omega,\mathbb U)\colon \B u\in\mathcal M(\Omega,\mathbb V)\},
    \end{gather*}
and we abbreviate $\bv^\B(\Omega)\coloneqq \bv^\B_p(\Omega)$ when $p=1$.
The spaces above are naturally endowed with weak-$\ast$ notions of convergence, namely
\begin{gather*}
    u_j \wstar u \text{ in } \bv^\B_p(\Omega)
    \quad\text{if and only if}\quad
    u_j \wk u \text{ in } L^p(\Omega)\text{ and }
    \B u_j \wstar \B u\text{ in } \mathcal M(\Omega,\mathbb V).
\end{gather*}
Stronger convergences may be retrieved if the class of differential operators is restricted.
We give a brief account on this point in the following lines.
 
Due to the interaction between Fourier transform and linear PDE,
often analytic properties of $\bv^\B$ spaces
(and of the equation $\B u=v$ in general)
can be expressed in terms of algebraic properties of the characteristic polynomial.
We recall that the characteristic polynomial, or symbol, of $\B$ is
    $$
     \B (\xi)\coloneqq \sum_{|i|=k}\xi^i B_i\in\lin(\U,\V),
     \qquad \xi\in\C^n,
    $$
where $\xi^i\coloneqq \xi_1^{i_1}\cdots \xi_n^{i_n}$.
In our study, the following property will be particularly relevant:
    \begin{definition}[\cite{Sm,BrDiGm,GmRa_emb}]
    \label{def:C-ell}
        An operator $\B$ is said to be $\C$-elliptic if 
        $$
        \ker_\C\B(\xi)=\{0\}\quad\text{ for all }\xi\in\C^n\setminus\{0\}.
        $$
    \end{definition}
    It was shown in \cite{Sm} that $\C$-ellipticity is equivalent with full Sobolev regularity for the equation $\B u=v$ on domains, provided that $v\in L^p(\Omega,\V)$, $p\in(1,+\infty)$.
    For $p=1$ we have the counterpart:
    \begin{theorem}[\cite{GmRa_emb}]\label{thm:GR}
       Let $\Omega\subset\R^n$ be a Lipschitz domain. An operator $\B$ is $\C$-elliptic if and only if 
       $$
       \|u\|_{\sobo^{k-1,n/(n-1)}(\Omega)}\leq c\left(|\B u|(\Omega)+\|u\|_{L^{1}(\Omega)}\right)\quad\text{for }u\in\bv^\B(\Omega),
       $$
     where $| \B u |$ is the total variation measure associated with $\B u$.
    \end{theorem}

\subsection{The bilevel scheme}
\label{sec:setup-bilevel}
We are now in a position to formulate our problem rigorously.
Let us fix $p\in(1,+\infty)$.
As we have touched upon in the introduction,
our goal is to provide an existence result for solutions to the following training scheme:
given $g\in L^p(\Omega,\U)$,
\begin{subequations}
	\begin{gather}
	\label{eq:T-L1}\tag{L1}
	\text{find } \alpha^\ast \in {\rm argmin} \left\{ F(u_\alpha) : \alpha \in \Adm \right\}\\
	\label{eq:T-L2}\tag{L2}
	\text{such that }u_\alpha\in {\rm argmin}\left\{I[u;\alpha]\colon u\in \bv_p^\B(\Omega)\right\},
	\end{gather}
\end{subequations}
where
\begin{equation}\label{eq:functional}
	I[u;\alpha] \coloneqq \Phi_g(u) 
	+ \int_{\Omega}\alpha(x) f(x,\dif \B u).
\end{equation}
 All the due definitions and assumptions are collected  below.

\begin{description}[itemsep=6pt,labelindent=\parindent,leftmargin=\parindent,font=\normalfont\itshape \space -- \space]
\item[Cost functional]
    As for the upper level problem \eqref{eq:T-L1},
    $F\colon L^p(\Omega,\U) \to \R$ is a proper, convex and weakly lower semicontinuous functional.
    Typical choices for this functional are the PSNR maximizing $F_{\PSNR}$ in \eqref{intro:ul_psnr}, which makes use of the ground truth $u_{\gt}$,
    and the statistics-based, ground truth-free $F_{\stat}$ in \eqref{intro:ul_stats}, in the spirit of supervised and unsupervised learning respectively.
\item[Fidelity term]
    The assumptions on the functional $\Phi_g\colon L^p(\Omega,\U) \to \R$ in \eqref{eq:functional} are similar to the ones on $F$,
    namely $\Phi_g$ is a proper, convex and weakly lower semicontinuous functional
    that is also coercive. This means that
    \[
        \lim_{j\to+\infty}\| u_j - g \|_{L^p(\Omega,\U)} = +\infty
        \quad\text{implies}\quad
        \lim_{j\to+\infty}\Phi_g(u_j) = +\infty.
    \]
    In particular,
    \[
        \Phi_g(u) = \| u_j - g \|^p_{L^p(\Omega,\U)}
    \]
    is a simple instance of fidelity term.
\item[Weights]
    Given $\underline \alpha,\overline \alpha\geq 0$ with $\underline{\alpha}<\overline{\alpha}$,
    the scalar fields $\alpha\in\hold(\bar{\Omega},[\underline \alpha,\overline \alpha])$ are supposed
    to share the same uniform modulus of continuity $\omega$,
    that is, an increasing function $\omega\colon[0,+\infty)\rightarrow[0,+\infty)$ such that $\omega(0)=0$.
    As a consequence, the class of admissible weights
    \begin{align}\label{eq:Adm}
        \Adm \coloneqq \left\{\alpha\in\hold(\bar{\Omega},[\underline \alpha,\overline \alpha]): |\alpha(x)-\alpha(y)|\leq\omega(|x-y|) \;\text{for every }x,y\in\overline{\Omega}\right\}
    \end{align}
    is compact with respect to the uniform norm by Arzel\`a--Ascoli theorem.
    We will motivate the definition of the set $\Adm$ below, see Subsection \ref{sec:why-Adm}.
\item[Integrand]
    The function $f\colon \Omega\times \V\rightarrow\R$ is a \emph{Carath\'eodory} integrand
    such that $z\mapsto f(x,z)$ is convex
    for $\mathscr{L}^n$-a.e. $x\in\Omega$. Here, {Carath\'eodory} integrand means jointly Borel measurable and continuous in the second variable.
    We also suppose that
    the integrand satisfies the linear coercivity and  growth bounds
    \begin{align}\label{eq:bounds}
        c( |\,\bigcdot\,|-1)\leq f(x,\,\bigcdot\,)\leq C(1+|\,\bigcdot\,|)\quad\text{for }\mathscr{L}^n\text{-a.e. }x\in\Omega,
    \end{align}
    for some $c,C\geq 0$.
\end{description}

In order to make sense of \eqref{eq:functional},
we are still left to define the applications of convex functions to measures,
as in the term $f(x,\dif \B u)$.
For an integrand $f\colon\Omega\times\V\rightarrow\R$ satisfying \eqref{eq:bounds},
we define the \emph{recession function}
    \begin{equation}\label{eq:recession}
    f^\infty(x,z)\coloneqq \lim_{(x',z',t)\rightarrow(x,z,+\infty)} \frac{f(x',tz')}{t}\quad\text{for }(x,z)\in\bar\Omega\times\V,    
    \end{equation}
which we assume to exist.
We then set for $\mu\in\mathcal M(\Omega,\V)$
    \begin{equation}\label{eq:action}
    \int_{\Omega}f(x,\dif \mu) \coloneqq\int_{\Omega}f\left(x,\dfrac{\dif \mu}{\dif\mathscr{L}^n}(x)\right)\dif x+\int_{\Omega}f^\infty\left(x,\dfrac{\dif\mu^s}{\dif|\mu|}(x)\right)\dif|\mu|(x),
    \end{equation}
where $\mu^s$ denotes the singular part of $\mu$ with respect to Lebesgue measure
and $\dif \mu / \dif \nu$ is the Radon-Nikod\'ym derivative of $\mu$ with respect to the measure $\nu$.

\subsection{Rationale for the definition of the set of admissible weights}\label{sec:why-Adm}
In order to highlight the main technical obstacles
that  are encountered in the analysis of bilevel training schemes with space-dependent weights,
we start with an example involving the weighted total variation,
which, in spite of its simplicity, exhibits the typical features of such class of problems.
The model we address has been already studied in \cite{HiRa17}
(with a different approach from the one we outline).

Let $\Omega\subset \R^n$ be a bounded open set with Lipschitz boundary.
For $p\in [1,\frac{n}{n-1})$,
we suppose that a training pair $(u_{\gt},g)\in L^2(\Omega,\R)\times L^p(\Omega,\R)$
is assigned,
where $u_\gt$ and $g$ encode respectively the ground truth and the corrupted datum.
We also fix two positive parameters $\underline \alpha$ and $\overline \alpha$,
and we provisionally allow the regularizing weights
to vary in $\mathrm{LSC}(\Omega,[\underline \alpha,\overline \alpha])$,
the space of lower semicontinuous functions on $\Omega$ with range in $[\underline \alpha,\overline \alpha]$.

For $u\in \bv(\Omega)$ and $\alpha\in \mathrm{LSC}(\Omega,\left[\underline \alpha,\overline \alpha\right])$,
we introduce the first order functional
\begin{equation}\label{eq:1st-order}
	J[u;\alpha]\coloneqq \int_\Omega |u-g|^p\dif x+\int_\Omega \alpha(x)\dif|\D u|(x),
\end{equation}
and the ensuing corresponding training scheme:
\begin{gather}
	\label{eq:T-L1-BV}
	\text{find } \alpha^\ast\in {\rm argmin}
	    \left\{F_{\PSNR}(u_\alpha)
	    : \alpha\in \mathrm{LSC}(\Omega,\left[\underline \alpha,\overline \alpha\right])
	    \right\}\\
	\label{eq:T-L2-BV}
	\text{such that } u_{\alpha}\in {\rm argmin}\left\{J[u;\alpha]\colon u\in \bv(\Omega)\right\}. 
\end{gather}

The functional in \eqref{eq:1st-order} is reminiscent of the one considered in \cite{AJNO17},
where, motivated by vortex density models,
the authors studied the property of minimizers,
i.e., of solutions to \eqref{eq:T-L2-BV}.

Before discussing the existence of solutions to the scheme \eqref{eq:T-L1-BV}--\eqref{eq:T-L2-BV} as a whole,
let us justify the choice of the class of weights in \eqref{eq:T-L1-BV}.
Note that
the definition of $J$ itself calls for some degree of regularity for $\alpha$.
Indeed, if in \eqref{eq:1st-order} $\alpha\colon \Omega \to [\underline \alpha,\overline \alpha]$ is a given function and $u$ is allowed to vary in $\bv(\Omega)$
(as it is the case of \eqref{eq:T-L2-BV}),
there might be choices of $u$ for which the coupling
\[ \int_\Omega \alpha(x)\dif|\D u|(x) \]
is not well-defined.
Prescribing lower semicontinuity for the admissible weights $\alpha$ allows to circumvent the issue,
{because lower semicontinuous functions are Borel measurable and $\D u\in \mathcal{M}(\Omega,\R^n)$ is a Borel measure.}
Besides, for any $\alpha\in \mathrm{LSC}(\Omega,\left[\underline \alpha,\overline \alpha\right])$ the existence of a solution $u_\alpha$ to \eqref{eq:T-L2-BV}
follows by the direct method of the calculus of variations.
Indeed, we firstly observe that
the coercivity of $J$ in $L^1$ is deduced
by the following standard result
(see e.g.\ \cite[Theorem 3.23]{AFP00}):
\begin{theorem}[Compactness in $\bv$]
    \label{stm:compactBV}
    Let $\Omega \subset \R^n$ be a bounded Lipschitz domain and
    let $(u_j)_j$ be a bounded sequence in $\bv(\Omega)$.
    Then, there exist $u\in\bv(\Omega)$ and a subsequence $(j_k)_k$ such that
    $(u_{j_k})$ weakly-$\ast$ converges to $u$, that is,
    $u_{j_k}\to u$ in $L^1(\Omega)$ and
    \begin{align*}
        \lim_{k\to+\infty} \int_{\Omega} \phi\, \dif \D u_{j_k} = \int_{\Omega} \phi \,\dif \D u
        \qquad\text{for all } \phi\in \hold_0(\Omega).
    \end{align*}
\end{theorem}
Secondly, we notice that
$J[\,\bigcdot\,;\alpha]$ is lower semicontinuous with respect to the $L^1$-convergence,
because,
when $\alpha\in \mathrm{LSC}(\Omega,\left[\underline \alpha,\overline \alpha\right])$,
general lower semicontinuity results in $\bv$ may be invoked
(see e.g.\ \cite{FoLe01}; and also \cite{AmDeFu08} for lower semicontinuity and relaxation results with $\bv$ integrands).

Once we know that
for lower semicontinuous weights, solutions to \eqref{eq:T-L2-BV} exist,
we can try to tackle the complete scheme.
So, let $(\alpha_j)_j \subset \mathrm{LSC}(\Omega,\left[\underline \alpha,\overline \alpha\right])$  be a minimizing sequence for \eqref{eq:T-L1-BV}.
Then, by definition, the integrals
	\[
		F_{\PSNR}(u_j)=\| u_j-u_\gt \|^2_{L^2(\Omega)}
		\qquad\text{with } u_j \coloneqq u_{\alpha_j}
	\]
converge, and we deduce that
$(u_j)_j$ is a bounded sequence in $L^2(\Omega)$.
Denote by $u \in L^2(\Omega)$ the weak $L^2$-limit of (a subsequence of) $(u_j)_j$.
By lower semicontinuity of the $L^2$-norm,
we obtain
\[
	F_{\PSNR}(u) \leq \inf \left\{ F_{\PSNR}(u_j) : \alpha\in \mathrm{LSC}(\Omega,\left[\underline \alpha,\overline \alpha\right]) \right\}.
\]
If we manage to show that $u=u_{\alpha^\ast}$ for some admissible $\alpha^\ast$,
then the latter is a solution to \eqref{eq:T-L1-BV}.
The natural choice for $\alpha^\ast$ would be the weak-$\ast$ limit of $(\alpha_j)_j$ in $L^\infty(\Omega)$, which can fail in general to have any lower semicontinuous representative.
On the positive side,
$J[u;\,\bigcdot\,]$ is continuous with respect to a suitable weak-$\ast$ convergence.
Indeed, if $(\alpha_j)_j \subset \mathrm{LSC}(\Omega,\left[\underline \alpha,\overline \alpha\right])$ is bounded and $u\in \bv(\Omega)$,
then there exist a subsequence, which we do not relabel, 
and $\alpha^\ast \in L^\infty(\Omega,[\underline \alpha,\overline \alpha];|\D u|)$ such that
    \begin{equation*}
	\lim_{j\to +\infty} \int_{\Omega} \alpha_j(x) \phi(x)  \dif|\D u|(x)  = \int_{\Omega} \alpha^\ast(x)  \phi(x)  \dif|\D u|(x)
		\qquad\text{for all }\phi\in L^1(\Omega;|\D u|).
	\end{equation*}
	In particular,
	\begin{equation}
	    \label{rmk:weakstar-cont}
		\lim_{j\to +\infty} J[u;\alpha_j]  = J[u;\alpha^\ast].
	\end{equation}

The previous lines suggest that
what is missing to solve the scheme \eqref{eq:T-L1-BV}-\eqref{eq:T-L2-BV} is a compactness property for the class of admissible weights.
This leads us to reduce ourselves to the problem
\begin{gather}
    \label{eq:T-L1-BV*}
    \begin{split}
	\text{find } \alpha^\ast\in {\rm argmin}\left\{F_{\PSNR}(u_\alpha) : \alpha\in \mathrm{Adm} \right\}\\
	\nonumber
    \text{ with }u_{\alpha}\in {\rm argmin}\left\{J[u;\alpha]\colon u\in \bv(\Omega)\right\}, 
    \end{split}
\end{gather}
where we assume \textit{a priori} that $\mathrm{Adm}\subset \hold(\bar\Omega,\left[\underline \alpha,\overline \alpha\right])$ is compact with respect to the uniform convergence.
Under the compactness assumptions on the class of admissible weights,
if $(\alpha_j)_j \subset \mathrm{Adm}$ is a minimizing sequence for \eqref{eq:T-L1-BV*},
and if $(u_j)_j$ and $u$ are constructed as above,
we are actually able to prove that $u = u_{\alpha^\ast}$,
where $\alpha^\ast\in \mathrm{Adm}$ is the uniform limit of $(\alpha_j)$.
In other words, the couple $(\alpha^\ast,u)$ is a solution to the scheme consisting of \eqref{eq:T-L1-BV*}--\eqref{eq:T-L2-BV}.

To prove the claim, we need to show that
\begin{equation}\label{eq:umin0}
J[u;\alpha^\ast] \leq J[v;\alpha^\ast]
\qquad\text{for any } v\in\bv(\Omega).
\end{equation}
We start from observing that the definition of $u_j$ grants
\[
	J[u_j;\alpha_j] \leq J[v;\alpha_j]
\qquad \text{for any } v\in\bv(\Omega),
\]
and hence, for any $v\in \bv(\Omega)$,
\begin{equation}\label{eq:finiteliminf0}
	\liminf_{j\to+\infty} J[u_j;\alpha_j]
	\leq \liminf_{j\to+\infty} J[v;\alpha_j]
	= J[v;\alpha^\ast],
\end{equation}
where the equality follows by \eqref{rmk:weakstar-cont}.
In particular, 
\[
	\liminf_{j\to+\infty} J[u_j;\alpha_j] <+\infty.
\]
Then, the uniform lower bound $ \alpha_j \geq \underline \alpha$ and Theorem \ref{stm:compactBV} yield that
$(u_j)$ converges weakly-$\ast$ in $\bv(\Omega)$ 
(again upon extraction of subsequences)
to a limit function which is necessarily $u$.
We thereby infer
\[
	u \in \bv(\Omega)\cap L^2(\Omega).
\]
From the uniform convergence of $(\alpha_j)$ and the weak-$\ast$ convergence of $(u_j)$ we obtain
\begin{equation}\label{eq:lim-uj-alphaj0}
	J[u;\alpha^\ast] \leq \liminf_{j\to+\infty} J[u_j;\alpha^{(j)}].
\end{equation}
On the whole, owing to \eqref{eq:finiteliminf0}, we deduce \eqref{eq:umin0}.

\begin{remark}
    In the absence of compactness for the set $\mathrm{Adm}$ under uniform convergence,
    the analysis becomes more delicate.
    We outline here some of the issues.
    
    Keeping in force the notation above,
    let $u\in \bv(\Omega)$ be the weak-$\ast$ limit of $(u_j)$
    and let $\alpha^\ast \in L^\infty(\Omega,[\underline \alpha,\overline \alpha];|\D u|)$ be the weak-$\ast$ limit of $(\alpha_j)$.
    Proving the optimality of $u$, i.e. $u=u_{\alpha^\ast}$, means
    \[
	J[u;\alpha^\ast] \leq J[v;\alpha^\ast]
	\qquad\text{for any } v\in \bv(\Omega).
    \]
    However, the right-hand side might be not well-defined.
    Intuitively, the point is that
    an ideal class of weights should be \textit{a priori} ``sufficiently compact'',
    and at the same it should give rise to ``well-behaved'' weighted $\bv$ functions.
    
    Another passage that is needed in the proof of existence (cf. \eqref{eq:2}, \eqref{eq:3})
    is the following semicontinuity inequality:
    \begin{equation*}
	J[u;\alpha^\ast] \leq \liminf_{j\to+\infty} J[u_j;\alpha^{(j)}].
    \end{equation*}
    Knowing that $(u_j)$ weakly-$\ast$ converges to $u$,
    its validity is undermined
    if only weak-$\ast$ convergence is available for the weights.
\end{remark}

\section{Existence theorems for the lower level problems}\label{sec:LSC}
We begin with a general lower semicontinuity result for convex integrands with rough $x$-dependence:

\begin{proposition}\label{prop:lsc}
Let $f\colon\Omega\times\V\rightarrow[0,+\infty)$ be a Borel measurable integrand such that $f(x,\,\bigcdot\,)$ is convex for almost every $x\in\Omega$.
Suppose that
the recession function $f^\infty$ in \eqref{eq:recession} exists for all $(x,z)\in\bar\Omega\times\V$.
Then 
$$
\mu_j\wstar \mu \text{ in }\mathcal M(\Omega,\V)\implies \liminf_{j\rightarrow\infty}\int_{\Omega}f(x,\dif\mu_j(x))\geq \int_{\Omega}f(x,\dif\mu(x)).
$$
\end{proposition}

Note that the restriction on the existence of the recession function of $f$ implies both the joint continuity of $f^\infty$ and the linear growth of $f$ from above.

\begin{proof}
We consider a Young measure $\ym$ generated by $(\mu_j)_j$. By Proposition \ref{prop:lsc_YM} and Jensen's inequality, we have that
\begin{align*}
\liminf_{j\rightarrow\infty} \int_{\Omega}f(x,\dif\mu_j(x))&\geq \int_{\Omega}\langle\nu_x,f(x,\,\bigcdot\,) \rangle\dif x+\int_{\bar\Omega} \langle\nu_x^\infty,f^\infty(x,\,\bigcdot\,) \rangle\dif\lambda\\
&\geq\int_{\Omega}f(x,\bar\nu_x) \dif x+\int_{\bar\Omega} f^\infty(x,\bar\nu_x^\infty) \dif\lambda\\
&= \int_{\Omega} \big[ f(x,\bar\nu_x)+\lambda^a(x)f^\infty(x,\bar\nu_x^\infty) \big]\dif x+\int_{\bar\Omega} f^\infty(x,\bar\nu_x^\infty) \dif\lambda^s
\end{align*}
Using the inequality $g(z)+tg^\infty(w)\geq g(z+tw)$
which holds for any convex function $g$ and $t\geq0$,
we have
\begin{align*}
\liminf_{j\rightarrow\infty} \int_{\Omega}f(x,\dif\mu_j(x))
&\geq\int_{\Omega}f(x,\bar\nu_x+\lambda^a(x)\bar\nu_x^\infty) \dif x+\int_{\bar\Omega} f^\infty(x,\bar\nu_x^\infty) \dif\lambda^s\\
&\geq \int_{\Omega}f(x,\mu^a(x)) \dif x+\int_{\Omega} f^\infty(x,\dif\mu^s) \\
&=\int_{\Omega}f(x,\dif\mu).
\end{align*}
where the inequality follows from the nonnegativity of $f$ and Lemma \ref{lem:barycentre}, which implies that $\mu=((\bar\nu_x+\mu^a(x))\mathscr L^n+\bar\nu_x^\infty\lambda^s)\mres\Omega $.
\end{proof}
\begin{remark}
As it can be seen from the proof of Proposition \ref{prop:lsc},
we can allow for signed integrands, as long as we consider $\mu_j\wstar\mu$ in $\mathcal{M}(\bar\Omega,\V)$. In this case, the recession term in \eqref{eq:action} should be an integral over $\bar\Omega$. 
\end{remark}

We next prove two general existence results for convex integrals defined on $\bv^\B_p$ spaces. 
The first one holds for \emph{arbitrary} operators $\B$.

\begin{theorem}\label{thm:lsc_any_B}
    Let us fix $p\in(1,+\infty)$, $g\in L^p(\Omega,\U)$, $\underline \alpha>0$ and $\alpha\in\hold(\bar\Omega, [\underline \alpha,+\infty))$.
    Let $f\colon\Omega\times\V\rightarrow[0,+\infty)$ be an integrand satisfying the assumptions outlined in Subsection \ref{sec:setup-bilevel},
    and suppose further that
    the recession function $f^\infty$ in \eqref{eq:recession} exists for all $(x,z)\in\bar\Omega\times\V$.
Then, the functional in \eqref{eq:functional}
is weakly-$\ast$ lower semicontinous and admits a minimizer $u\in\bv^\B_p(\Omega)$.
If the fidelity term is strictly convex,
then the minimizer is unique.
\end{theorem}
\begin{proof}
There is no loss of generality in assuming that $\alpha\equiv 1$.
We will employ the direct method in the calculus of variations.
Since $I$ is bounded from below,
there exists a minimizing sequence $(u_j)_j\subset \bv^\B_p(\Omega)$
such that the limit of $I[u_j;\alpha]$ as $j\to+\infty$ is finite.
Since $\Phi_g$ is coercive on $L^p(\Omega,\mathbb U$ and $f$ satisfies the growth condition in \eqref{eq:bounds},
$(u_j)_j$ must be bounded in $\bv^\B_p(\Omega)$.
Thus, on a subsequence that we do not relabel,
we have $u_j\rightharpoonup u$ in $L^p(\Omega,\U)$ and $\B u_j\wstar\B u$ in $\mathcal M(\Omega,\V)$.
By the weak lower semicontinuity of the fidelity term
we have that
\begin{align}\label{eq:0order}
    \Phi_g(u)\leq \liminf_{j\to+\infty} \Phi_g(u_j),
\end{align}
whereas by Proposition \ref{prop:lsc} we obtain
\begin{align}\label{eq:korder}
    \int_{\Omega} f(x,\dif \B u)\leq \liminf_{j\rightarrow\infty}  \int_{\Omega} f(x,\dif \B u_j).
\end{align}
On the whole, we deduce
$$
    I[u;\alpha]\leq \liminf_{j\to+\infty} I[u_j;\alpha],
$$
and we conclude that $u\in\bv^\B_p(\Omega)$ is a minimizer of $I$.

Uniqueness follows easily when $\Phi_g$ is strictly convex.
Indeed, let $u^1,\,u^2$ be distinct minimizers and $u^0\coloneqq (u^1+u^2)/2$. Then, $2\Phi_g(u^0)<\Phi_g(u^1)+\Phi_g(u^2)$,
while the convexity of the second term gives
$$
2\int_{\Omega} f(x,\dif \B u^0)\leq \int_{\Omega} f(x,\dif \B u^1)+\int_{\Omega} f(x,\dif \B u^2).
$$
By adding the last two inequalities
we infer $I[u^0;\alpha]<\min I$, a contradiction.
\end{proof}

\begin{remark}
Notably, in the previous theorem uniqueness holds for $\Phi_g(u)=\| u - g \|^p_{L^p(\Omega)}$.
Indeed, for instance by the uniform convexity of the $L^p$ spaces, we have for $u^0\coloneqq (u^1+u^2)/2$ 
$$
2\int_\Omega |u^0-g|^p\dif x<\int_\Omega |u^1-g|^p\dif x+\int_\Omega |u^2-g|^p\dif x.
$$
\end{remark}

\begin{remark}
    There is no immediate counterpart of Theorem \ref{thm:lsc_any_B} when $p=1$,
    because in general bounded sequences in $\bv^\B$ are not weakly-$\ast$ precompact.
    One possibility would be to embed $\bv^\B$
    in the larger space of measures $\{\mu \in \mathcal{M}(\Omega,\U) : \B \mu \in \mathcal{M}(\Omega,\V)\}$.
    A second option is to assume $\B$ to be $\C$-elliptic,
    as we do below.
\end{remark}

The second existence result involves the smaller class of $\C$-elliptic operators,
which was introduced in Definition \ref{def:C-ell}.
In this case, we are able to treat regularizers that also involve lower order terms, see \eqref{our_reg},
and we can obtain much more precise information on the minimizers.
We make the unconventional convention that $\frac{n}{n-k}=+\infty$ if $k\geq n$, and
we denote by $\sym^i(\R^n,\U)$ the space of symmetric $\U$-valued $i$-linear maps on $\R^n$.

\begin{theorem}\label{thm:lsc_C-ell}
    Let us fix $p\in[1,\frac{n}{n-k})$, $g\in L^p(\Omega,\U)$, $\underline \alpha>0$,
    $\alpha_i\in\hold(\bar\Omega, [0,+\infty))$ for $i=1,\ldots,k-1$
    and $\alpha_k\in\hold(\bar\Omega, [\underline \alpha,+\infty))$.
    Let $f_i\colon\Omega\times\sym^i(\R^n,\U)\rightarrow\R$, $i=1,\ldots,k-1$, and $f_k\colon \Omega\times\V\rightarrow[0,+\infty)$ be Carath\'eodory integrands
    such that for all $i$
    $f_i(x,\,\bigcdot\,)$ is convex and $f_i(x,\,\bigcdot\,)\leq C(1+|\,\bigcdot\,|)$ for almost every $x\in\Omega$.
    Suppose in addition that
    $f_k$ satisfies the coercivity bound $c(|\,\bigcdot\,|-1)\leq f_k(x,z)$ and that
    $$
    f_k^\infty(x,z)\coloneqq \lim_{(x',z',t)\rightarrow(x,z,+\infty)}\frac{f_k(x',tz')}{t}
    $$
    exists for all $(x,z)\in\bar\Omega\times\V$.
    Then, if $\B$ is $\C$-elliptic, the functional 
    \begin{align}\label{eq:Itilde}
	    \tilde I[u;\alpha] \coloneqq \Phi_g(u) +\sum_{i=1}^{k-1}\int_{\Omega} \alpha_i(x) f_i(x, \nabla^i u(x))\dif x
	+ \int_{\Omega}\alpha_k(x) f_k(x,\dif \B u).
    \end{align}
    is weakly-$\ast$ lower semicontinuous in $\bv^\B$ and  
    admits a minimizer $$u\in\bv^\B(\Omega)\cap\sobo^{k-1,n/(n-1)}(\Omega,\U).$$
    If the fidelity term is strictly convex,
    then the minimizer is unique. 
    \end{theorem}
\begin{proof}
If $k=1$ the statement collapses to Theorem \ref{thm:lsc_any_B}.
The $\C$-ellipticity of $\B$ is still needed to make use of Theorem~\ref{thm:GR},
which grants that $u\in L^{n/(n-1)}(\Omega,\U)$.
We now turn to the case $k\geq 2$. 

If $(u_j)_j\subset\bv^\B(\Omega)$ is a minimizing sequence,
then $(u_j)_j$ is bounded in $\bv^\B(\Omega)$ by the coercivity assumptions, and hence also in $\sobo^{k-1,n/(n-1)}(\Omega,\U)$ thanks to Theorem~\ref{thm:GR}.
Let $u\in\bv^\B$ be a weak-$\ast$ limit point of $(u_j)_j$.
By the same reasoning as in the proof of Theorem~\ref{thm:lsc_any_B}
we have that \eqref{eq:0order} and \eqref{eq:korder} with $f=f_k$ hold.
We now fix $1\leq i\leq k-1$ and look at a Young measure $\ym$ generated by $(\nabla^iu_j)_j$, which is bounded in $L^{n/(n-1)}(\Omega)$. By the growth bound on $f_i$ and the de la Vall\'ee Poussin criterion, we have that $(\alpha_if_i(\,\bigcdot\,,\nabla^iu_j))_j$ is uniformly integrable. We can thus employ Proposition~\ref{prop:cts_YM} for $f=\alpha_if_i$ to obtain
\begin{align*}
    \liminf_{j\rightarrow\infty}\int_{\Omega}\alpha_i(x)f_i(x,\nabla^iu_j(x))\dif x&= \int_{\Omega}\alpha_i(x)\langle\nu_x, f_i(x,\,\bigcdot\,)\rangle\dif x \\ 
    & \geq\int_{\Omega} \alpha_i(x)f_i(x,\bar\nu_x)\dif x\\
    &=\int_{\Omega}\alpha_i(x)f_i(x,\nabla^iu(x))\dif x,
\end{align*}
where we used Jensen's inequality and Lemma~\ref{lem:barycentre}.
We infer that
$$ \liminf_{j\to+\infty} \tilde I[u_j;\alpha]\geq \tilde I[u;\alpha],$$
and we can conclude that  $u\in\bv^\B(\Omega)\subset\sobo^{k-1,n/(n-1)}(\Omega,\U)$ is a minimizer of $\tilde I[\,\bigcdot\,,\alpha]$.

The uniqueness follows exactly by the same argument as in the proof of Theorem~\ref{thm:lsc_any_B},
so the conclusion is achieved.
\end{proof}

\begin{remark}
   {
   If  $\Phi_g(u)=\| u - g \|_{L^1(\Omega)}$,
   uniqueness might fail in Theorem \ref{thm:lsc_C-ell}}.
\end{remark}

\section{The bilevel training scheme in the space $\bv^\B_p$}\label{sec:bilevel}

We devote this section to the proof of our main theoretical result, that is,
the existence of solutions to the bilevel scheme \eqref{eq:T-L1}--\eqref{eq:T-L2}.
The study of the lower level problem will be addressed by Theorem \ref{thm:lsc_any_B}.
A variant involving functionals as in Theorem
\ref{thm:lsc_C-ell} will also be presented, see Remark \ref{thm:main2}.

    \begin{theorem}\label{thm:main}
	    Let us fix $p\in(1,+\infty)$, $g\in L^p(\Omega,\U)$, $\underline \alpha>0$ and $\alpha\in\hold(\bar\Omega, [\underline \alpha,+\infty))$.
        Let $f\colon\Omega\times\V\rightarrow[0,+\infty)$ be an integrand satisfying the assumptions outlined in Subsection \ref{sec:setup-bilevel},
        and suppose further that
        the recession function $f^\infty$ in \eqref{eq:recession} exists for all $(x,z)\in\bar\Omega\times\V$.
        Then, the training scheme
        \eqref{eq:T-L1}--\eqref{eq:T-L2} in Subsection \ref{sec:setup-bilevel}
        admits a solution $\alpha^\ast\in {\rm Adm}$ and
        it provides an associated optimally reconstructed image $u_{\alpha^\ast}\in\bv^\B_p(\Omega)$.
    \end{theorem}
    \begin{proof}
    Let $(\alpha_j)_j\subset\Adm$ be a minimizing sequence
    for the upper level objective $F$.
    Under our assumptions on the admissible weights,
    we may suppose that $\alpha_j\rightarrow \alpha^\ast\in\Adm$ uniformly in $\bar\Omega$.
    To prove the result, it suffices to show that
    \begin{align}\label{eq:alpha_star}
        F(u_{\alpha^\ast})\leq \lim_{j\rightarrow+\infty}F(u_j),
    \end{align}
    where we abbreviated $u_j\coloneqq u_{\alpha_j}$
    for a minimizer of \eqref{eq:T-L2} with respect to the weight $\alpha_j$,
    which exists in the light of Theorem \ref{thm:lsc_any_B}.
   
   We firstly show that $(u_j)$ is weakly-$\ast$ precompact in $\bv^\B_p(\Omega)$.
   To see this, we observe that
   by definition of $u_j$ we have
    \begin{align} \label{eq:1bis}
    I[u_j;\alpha_j]\leq I[v;\alpha_j]
    \quad \text{for any } v\in\bv^\B_p(\Omega).
    \end{align}
   In particular,
   by selecting $v=0$ and recalling that $\|\alpha_j\|_{L^\infty}\leq \overline \alpha$,
   we find that $I[u_j;\alpha_j]\leq C$ for some $C\geq 0$ independent of $j$.
   Then, owing to the coercivity of $I$,
   $(u_j)_j$ is bounded in $\bv^\B_p(\Omega)$ and
   there exists $u\in\bv^\B_p(\Omega)$ such that,
   upon extraction of subsequences,
   $u_j \wstar u$ in $\bv^\B_p(\Omega)$.
    If we prove that $u=u_{\alpha^\ast}$,
    the conclusion is then achieved,
    since \eqref{eq:alpha_star} would then follow by the lower semicontinuity of $F$.
   
    We thus need to show that
    \begin{align}\label{eq:0}
    I[u;\alpha^\ast]\leq I[v;\alpha^\ast]\quad\text{for any }v\in\bv^\B_p(\Omega).
    \end{align}
    
    The uniform convergence of $(\alpha_j)_j$ along with \eqref{eq:1bis} yields
    \begin{align}\label{eq:1}
    \liminf_{j\rightarrow+\infty} I[u_j;\alpha_j]\leq \liminf_{j\rightarrow+\infty}I[v;\alpha_j]=I[v;\alpha^\ast]
    \quad \text{for any } v\in\bv^\B_p(\Omega).
    \end{align}
    Further, in view of the growth condition \eqref{eq:bounds}
    and of the bound of $(u_j)_j$ in $\bv^\B_p(\Omega)$
    we obtain the estimates
    \begin{align*}
        |I[u_j;\alpha_j]-I[u_j;\alpha^\ast]| & \leq \int_\Omega|\alpha_j-\alpha^\ast|f(x, \dif \B u_j) \\ 
        & \leq C(1+|\B u_j|(\Omega))\|\alpha_j-\alpha^\ast\|_{L^\infty} \\
        & \leq C \|\alpha_j-\alpha^\ast\|_{L^\infty},
    \end{align*}
    whence
    \begin{align}\label{eq:2}
        \lim_{j\to+\infty} \big| I[u_j;\alpha_j]-I[u_j;\alpha^\ast] \big| = 0.
    \end{align}
    Finally,
    by the lower semicontinuity result in Theorem \ref{thm:lsc_any_B},
    \begin{align}\label{eq:3}
    I[u;\alpha]\leq \liminf_{j\rightarrow\infty} I[u_j;\alpha^\ast],
    \end{align}
    so that, collecting \eqref{eq:1}--\eqref{eq:3}, we obtain \eqref{eq:0}.
    The proof is thus complete.
\end{proof}

If in the lower level problem \eqref{eq:T-L2-BV} the functional $I$ is replaced by $\tilde I$ as in Theorem \ref{thm:lsc_C-ell},
a result in the same spirit of the one above holds.
We only sketch it in the next remark, since it parallels closely Theorem \ref{thm:main}.

    \begin{remark}\label{thm:main2}
    Within the general framework of Subsection \ref{sec:setup-bilevel}, 
    we introduce a variant of the scheme \eqref{eq:T-L1}--\eqref{eq:T-L2}.
    For $k\in\mathbb{N}$, $k\geq 2$,
    we define the sets
    \begin{gather*}
    \Adm_{\mathrm{low}} \coloneqq \left\{
        \alpha\in\hold(\bar{\Omega},[0,\overline \alpha]^{k-1})
        : |\alpha_i(x)-\alpha_i(y)|\leq\omega(|x-y|) \;\text{for every } i \text{ and }x,y\in\overline{\Omega}
        \right\}, \\ 
    \widetilde \Adm \coloneqq \Adm_{\mathrm{low}} \times \Adm,
    \end{gather*}
    where $\omega$ is the same modulus of uniform continuity as in \eqref{eq:Adm}.
    We consider the bilevel problem
    \begin{gather}
	\label{eq:T-L1-Cell}
	\text{find } \alpha^\ast \in {\rm argmin} \left\{ F(u_\alpha) : \alpha \in \widetilde\Adm \right\}\\
	\label{eq:T-L2-Cell}
	\text{such that }u_\alpha\in {\rm argmin}\left\{\tilde I[u;\alpha]\colon u\in \bv^\B(\Omega)\right\},
	\end{gather}
    where $\tilde I$ is as in \eqref{eq:Itilde}.
    Under the assumptions of Theorem \ref{thm:lsc_C-ell},
    notably $\C$-ellipticity for $\B$,
    we are able to prove the existence of a solution, that is,
    an optimal regularizer $\alpha^\ast\in\widetilde\Adm$ for \eqref{eq:T-L1-Cell}.
    Let us outline the argument.
    
    If $(\alpha^j)_j\subset \widetilde\Adm$ is a minimizing sequence,
    we may assume that $\alpha^j\to\alpha^\ast\in\widetilde\Adm$ uniformly.
    By Theorem \ref{thm:lsc_C-ell} we can pick a sequence $(u_j)_j\subset \bv^\B(\Omega)$
    made of minimizers for \eqref{eq:T-L2-Cell} associated with $(\alpha^j)_j$.
    As a consequence of the coercivity of $\tilde I$,
    $(u_j)_j$ is bounded in $\bv^\B_p(\Omega)$,
    and thus, owing to Theorem \ref{thm:GR}, also in $\sobo^{k-1,n/(n-1)}(\Omega)$.
    Denoting by $u\in\bv^\B(\Omega)$ the weak-$\ast$ limit (up to subsequences) of $(u_j)_j$,
    the remainder of the proof follows the one of Theorem \ref{thm:main2},
    the most significant difference being the use
    of Theorem \ref{thm:lsc_C-ell} instead of Theorem \ref{thm:lsc_any_B}
    to obtain the analogue of \eqref{eq:3}.
    \end{remark}

\section{Numerical examples}\label{sec:num}
In this section we provide some numerical results for image reconstruction
by focusing on some specific instances of the differential operators considered above. 
These numerical examples show the applicability and versatility of our approach,
which, as we will see, is able to yield results 
that are comparable, and in certain cases even better,
than the ones obtained by using some standard high quality regularizers,
such as the Total Generalized Variation (TGV) \cite{TGV} and its weighted version \cite{bilevelTGV}.
Since our main target here is to evaluate the performance
of the types of regularizers that we introduced,
we restrict ourselves to two particular cases of image denoising.
Firstly, in the class of first-order functionals, 
we consider a Huber-type TV regularization,
with both the regularization parameter $\alpha$ and the Huber parameter $\gamma$  being spatially dependent.
This can be considered as a functional
that incorporates a local choice between  TV and  Tikhonov regularization.
The second example is a spatially varying TV$^{2}$ regularization,
which is a second-order functional and
has the capability to improve the reconstructions
by eliminating the undesirable staircasing effect of TV \cite{PaSc}.
Even though in theory the TV$^{2}$ regularization is not able to preserve sharp edges,
we will see that its spatially varying  version produces high quality results
and can even outperform both the scalar and the spatially varying versions of TGV.

\subsection{Weighted Huber versions of  TV and TV$^{2}$}
Let $\gamma\in L^{\infty}(\Omega)$, with $\gamma\ge 0$, be fixed.
We define the spatially varying Huber function
$f_{\gamma}\colon \Omega \times \mathbb{R}^{n} \to \mathbb{R}$ as follows:
\begin{equation}\label{f_gamma}
f_{\gamma}(x,z)=
\begin{cases}
|z|-\dfrac{1}{2} \gamma(x), & \text { if } |z|\ge \gamma(x),\\[0.8em]
\dfrac{1}{2\gamma(x)} |z|^{2}, & \text { if } |z| < \gamma(x).
\end{cases}
\end{equation}
Obviously, for all $z\in \R^{n}$ and for almost all $x\in \Omega$,
$f_{\gamma}$ satisfies the coercivity and growth conditions in \eqref{eq:bounds}, namely
\begin{equation}\label{f_gamma_lg}
|z|-\frac{1}{2} \|\gamma\|_{L^{\infty}(\Omega)}\le f_{\gamma}(x,z) \le |z|. 
\end{equation}
Then, if $u\in \mathrm{BV}(\Omega)$, we define the ensuing convex function of the measure $\D u$ with the  alternative notations
\[\mathrm{TV}_{\gamma}(u)\coloneqq |f_{\gamma}(\D u)|(\Omega)\coloneqq \int_{\Omega} f_{\gamma}(x, d\D u).\]
An easy check shows that
the recession function of $f_\gamma$ (cf. \eqref{eq:recession}) is $f_\gamma^\infty(x,z)=|z|$. Thus, all the assumptions of Theorem \ref{thm:main} are trivially satisfied.
Consequently, $\mathrm{TV}_{\gamma}$ is indeed well-defined  as
\[\mathrm{TV}_{\gamma}(u)= \int_{\Omega} f_{\gamma}(x, \nabla u) \dif x + \int_{\Omega} d|\D^{s}u|, \]
and for $\alpha\in C(\overline{\Omega})$ with $\alpha(x)\ge\underline{\alpha}>0$ we can define its weighted version
\begin{equation}\label{weightedHuberTV_primal}
\mathrm{TV}_{\alpha, \gamma}(u)= \int_{\Omega} \alpha f_{\gamma}(x, \nabla u) \dif x + \int_{\Omega} \alpha\, d|D^{s}u|. 
\end{equation}
 Note that $\mathrm{TV}_{\alpha, \gamma}(u)$ can be equivalently defined via duality \cite{Dem}:
 \begin{equation}\label{weightedHuberTV_dual}
\mathrm{TV}_{\alpha, \gamma}(u)=\sup \left \{\int_{\Omega} u\,\mathrm{div} \phi\, \dif x  - \mathcal{I}_{\{|\,\bigcdot\, (x)|\le \alpha(x)|\}}(\phi)-\frac{1}{2} \int_{\Omega} \frac{\gamma}{\alpha} |\phi|^{2}\, \dif x:\; \phi\in \ccinfty(\Omega, \mathbb{R}^{n})  \right \}.
 \end{equation}
This is an immediate extension of the standard Huber TV functional \eqref{intro:hTV_dual} mentioned in the introduction. Similarly, for a function $u\in \mathrm{BV}^{2}(\Omega)\coloneqq\{u\in W^{1,1}(\Omega): D^{2}u \in \mathcal{M}(\Omega, \mathcal{S}^{n\times n})\}$ and $\alpha\in \hold(\overline{\Omega})$ with $\alpha(x)\geq \underline\alpha>0$, we define the weighted Huber TV$^{2}$ functional as
\begin{equation}\label{weightedHuberTV_primal}
\mathrm{TV}_{\alpha, \gamma}^{2}(u)= \int_{\Omega} \alpha f_{\gamma}(x, \nabla^{2} u)\, \dif x + \int_{\Omega} \alpha\, \dif |(\D^{2})^{s}u|,
\end{equation}
where $f_{\gamma}$ now is a function defined on $\Omega\times \R^{n\times n}$
defined by the natural analogue of \eqref{f_gamma}.
Again via duality, we find the equivalent expression
 \begin{equation}\label{weightedHuberTV_dual}
\mathrm{TV}_{\alpha, \gamma}^{2}(u)=\sup \left \{\int_{\Omega} u\,\mathrm{div}^{2} \phi\, \dif x  - \mathcal{I}_{\{|\,\bigcdot\, (x)|\le \alpha(x)\}}(\phi)-\frac{1}{2} \int_{\Omega} \frac{\gamma}{\alpha} |\phi|^{2}\, \dif x:\; \phi\in \ccinfty(\Omega, \mathbb{R}^{n})  \right \}.
 \end{equation}
 
Our examples concern the following lower level image denoising problems:
\begin{align}
I_{\gamma}^{1}[u;\alpha]
&\coloneqq\int_{\Omega}(u-g)^{2}\dif x + \mathrm{TV}_{\alpha ,\gamma}(u), \label{l1}\\
I_{\gamma}^{2}[u;\alpha]
&\coloneqq\int_{\Omega}(u-g)^{2}\dif x +  \mathrm{TV}_{\alpha, \gamma}^{2}(u). \label{l2}
\end{align}

\subsection{The bilevel problems}
The family of bilevel problems for the automatic computation of the spatial regularization parameter $\alpha$ associated with the functional $I_{\gamma}^{i}$ for $i=1,2$ is:
	\begin{gather}
	\label{bilevel_Huber1}
	\text{find }\alpha^\ast\in \underset{}{ \rm argmin} \left \{ F(u_{\alpha}): \; \alpha\in \mathrm{Adm} \right \}\\
	\label{bilevel_Huber2}
	\text{ such that  }u_{\alpha}= \underset{}{ \rm argmin}\left\{I_{\gamma}^{i}[u;\alpha]\colon u\in \bv^{i}(\Omega)\right\}.
	\end{gather}

In view of \eqref{f_gamma_lg}, the lower level problems \eqref{bilevel_Huber2} are well-defined,
while from Theorem \ref{thm:main} we know that
the overall schemes \eqref{bilevel_Huber1}--\eqref{bilevel_Huber2} admit a solution for the two alternative upper level objectives $F$ considered next. As we discussed in the introduction and repeat here for the sake of reading flow, we take into account two alternatives for the upper level objective functional $F$:
\begin{align}
 F_{\PSNR}(u)&= \int_{\Omega} |u-u_{\gt}|^{2}\dif x,\label{Fpsnr}\\
  F_{\stat}(u)&=\frac{1}{2} \int_{\Omega} \max(Ru-\overline{\sigma}^{2},0)^{2} \dif x + \frac{1}{2} \int_{\Omega} \min (Ru-\underline{\sigma}^{2},0)^{2} \dif x, \label{Fstat}
  \end{align}
 \[\text{where } Ru(x)\coloneqq\int_{\Omega} w(x,y) (u-g)^{2}(y) \dif y\;\; \text{ for } w\in L^{\infty}(\Omega\times \Omega), \text{  } \int_{\Omega}\int_{\Omega} w(x,y)\dif x\dif y=1.\]
The first cost functional corresponds to a maximization of the PSNR of the reconstruction and requires the knowledge of the ground truth $u_{\gt}$ \cite{bilevellearning, TGV_learning2, GVlearningfirst, noiselearning, pockbilevel}, while the second enforces the localized residuals $R u$ to belong in a certain tight corridor $[\underline{\sigma}^{2}, \overline{\sigma}^{2}]\coloneqq [\sigma^{2}-\epsilon, \sigma^{2} +\epsilon]$, $\sigma^{2}$ being the variance of the  noise $\eta$, which is assumed here to be Gaussian, see also \cite{HiRa17, hintermuellerPartII, bilevel_handbook, bilevelTGV}. The latter option has the advantage of being ground truth free, but knowledge or a good estimate for the noise variance $\sigma^{2}$ is needed. For the discrete version of the averaging filter $w$ in the definition of the localized residuals \eqref{Fstat} we use a filter of size $n_{w} \times n_{w}$, with entries of equal value that sum to  one.

Since a numerical projection to the admissible set $\mathrm{Adm}$ is not practical, here we also follow \cite{HiRa17, hintermuellerPartII, bilevel_handbook, bilevelTGV} and add instead a small $H^{1}$ term of the weight function $\alpha$ in the upper level objective, together with a supplementary box constraint $\mathcal{C}\coloneqq\{\alpha\in H^{1}(\Omega): \underline{\alpha} \le \alpha \le \overline{\alpha}\}$ for some $\underline{\alpha}, \overline{\alpha}\in \R$ with $0<\underline{\alpha}<\overline{\alpha}$. On the whole, we will use the following upper level objectives:
\begin{align*}
 \mathcal{F}_{\PSNR}(\alpha, u)&= \int_{\Omega} |u-u_{\gt}|^{2}\dif x + \frac{\lambda}{2} \|\alpha\|_{H^{1}(\Omega)}^{2},\\
  \mathcal{F}_{\stat}(\alpha, u)&=\frac{1}{2} \int_{\Omega} \max(Ru-\overline{\sigma}^{2},0)^{2}\dif x + \frac{1}{2} \int_{\Omega} \min (Ru-\underline{\sigma}^{2},0)^{2}\dif x + \frac{\lambda}{2} \|\alpha\|_{H^{1}(\Omega)}^{2},
\end{align*}
for some small $\lambda>0$. We will denote by $\hat{\mathcal{F}}$ the corresponding reduced objective functionals, that is $\hat{\mathcal{F}}_{\PSNR/\stat} (\alpha)\coloneqq \mathcal{F}_{\PSNR/\stat} (\alpha, u_{\alpha})$. That leads us to the bilevel minimization problems that we  tackle numerically: for $i=1,2$
\begin{gather}
\text{find }\alpha^\ast\in \underset{\alpha}{ \rm argmin}\;  \mathcal{F}_{\PSNR/\stat}(\alpha, u_{\alpha})\label{bilevels_final1}\\
\text{such that  }u_{\alpha}= \underset{u}{ \rm argmin}\;I_{\gamma}^{i}[u;\alpha] \;\; \text{ and } \l\l \alpha\in \mathcal{C}. \label{bilevels_final2}
\end{gather} 

 Note that in this setting it is not guaranteed that $\alpha\in C(\overline{\Omega})$,
 since $H^{1}(\Omega)$ does not embed in that space for dimensions higher than $1$.
 However, one can take advantage of a regularity result of the $H^{1}$-projection onto $\mathcal{C}$, denoted by $P_{\mathcal{C}}$ see \cite[Corollary 2.3]{hintermuellerPartII}. This projection is  applied to every iteration of the projected gradient algorithm,  which is to be used for the numerical solution of \eqref{bilevels_final1}--\eqref{bilevels_final2} and is described next, see Algorithm \ref{alg:bilevelTVandTV2}. In that case, it is ensured that the computed weight $\alpha^{k}$ at the $k$-th projected gradient iteration belongs to $H^{2}(\Omega)$, which for $n=2$ embeds compactly into any H\"older space $C^\beta(\overline{\Omega})$, $\beta\in(0,1)$.

\subsection{Strategy for fixing $\boldsymbol{\gamma}$} Since in our set up the function $\gamma$ is not part of the minimizing variables, it has to be fixed from the start. Our rationale for fixing $\gamma$ is that we would like to regularize high detailed areas with a weighted Tikhonov term $\frac{1}{2}\int_{\Omega} \tilde{\alpha} |\nabla u|^{2}\dif x$, with $\tilde{\alpha}$ having as low regularity as possible, e.g.\ $L^{\infty}(\Omega)$, in order to increase flexibility in the regularization. In the other areas we would like to regularize using a  weighted TV or TV$^{2}$ term with a spatially varying weight $\alpha$. This will happen if $\gamma$ is large in such detailed areas in order to allow for the second case in \eqref{f_gamma} and small otherwise. We thus adopt the following strategy: We first solve an auxiliary bilevel problem with a weighted Tikhonov regularizer using the upper level objective $\mathcal{F}_{\stat}$. The output is a weight $\tilde{\alpha}$ that essentially acts as an edge detector, since it is small on the edges and on the detailed areas of the image.
We then invert this weight and set
\begin{equation}\label{inv_alphatilde}
\gamma=s \frac{1}{\tilde{\alpha}}
\end{equation}
for some constant $s>0$.
 By choosing the function $\gamma$  as in \eqref{inv_alphatilde}, we have that when $\tilde{\alpha}$ is small (fine scale details), $\gamma$ will be large and thus the second case in \eqref{f_gamma} will be selected with a weight $\frac{1}{2 \gamma (x)}=\frac{s}{2}\tilde{\alpha}$ in front of the term $|\nabla u|^{2}$. On the other hand, when $\tilde{\alpha}$ is large, then $\gamma$ will be small and thus a TV or TV$^{2}$ term will be preferred, i.e., first case in \eqref{f_gamma}. In the third images of the top rows of Figures \ref{fig:parrot_stat} and \ref{fig:hatchling_stat}, we see  how the resulting $\gamma$ function looks like for the example images. Details for the computation of $\tilde{\alpha}$ via the auxiliary bilevel Tikhonov problem are given in the next section. We note however that, instead of solving a bilevel weighted Tikhonov problem to compute $\tilde{\alpha}$ and hence $\gamma$, one could alternatively employ some standard edge detector algorithms, like for instance the Canny method \cite{canny}.

\subsection{Numerical algorithm for solving the bilevel problems}
In this section we describe the  algorithm that we use for the numerical solution of the discrete versions of the bilevel problems \eqref{bilevels_final1}--\eqref{bilevels_final2}. Similar algorithms were presented in \cite{hintermuellerPartII, bilevelTGV}, so we limit ourselves to a brief description of the procedure, still providing all the necessary details to ensure reproducibility. 

The lower level problems \eqref{l1} and \eqref{l2} are substituted by their primal-dual optimality conditions, see \cite{stadler,bilevelTGV}:
\begin{align}
u-g-\mathrm{div} p&=0,\label{l1_res1}\\
\max(|\nabla u|, \gamma)p-\alpha\nabla u&=0,\label{l1_res2}
\end{align}
and
\begin{align}
u-g+\mathrm{div}^{2} p&=0,\label{l2_res1}\\
\max(|\nabla^{2} u|, \gamma)p-\alpha\nabla^{2} u&=0,\label{l2_res2}
\end{align}
with $p$ denoting the correspoding dual variable of each problem. For notation ease, we compactly write the above equations as $G_{1}(u,p)=0$ and $G_{2}(u,p)=0$. The application of the $\max$ function as well the multiplication in \eqref{l1_res2} and \eqref{l2_res2} are regarded component wise; note that here both $\alpha$ and $\gamma$ are spatially (i.e. pixel) dependent. We use standard forward and backward differences for the discretizations of  $\nabla$ and $\mathrm{div}$, see e.g.\ \cite{bilevelTGV}, and similarly for the discretizations of $\nabla ^{2}$ and $\mathrm{div}^{2}$, see \cite{PaSc}. For the numerical solution of \eqref{l1_res1}--\eqref{l1_res2}, we use a semismooth Newton algorithm as it is described in \cite{bilevelTGV} for the TGV case. Note that we do not add an additional Laplacian term for $u$ as in \cite{stadler, bilevelTGV}, and we do not smooth the $\max$ function. We terminate the semismooth Newton iterations when the Euclidean norm of both residuals is less than $10^{-4}$.

In order to solve the   minimization problems in \eqref{bilevels_final1}--\eqref{bilevels_final2}, where the lower level problems are substituted by  \eqref{l1_res1}--\eqref{l1_res2} and \eqref{l2_res1}--\eqref{l2_res2}, we employ a discretized projected gradient approach with Armijo line search as it is described in \cite{bilevelTGV}, originated from \cite{hintermuellerPartII}. The algorithm is summarized in Algorithm \ref{alg:bilevelTVandTV2}. We comment on the components that have not been clarified so far. The term $\Delta_{N}$ denotes the discrete Laplacian with zero Neumann boundary conditions. These are the desired boundary conditions for $\alpha$, as dictated by the regularity results for the $H^{1}$-projection $P_{\mathcal{C}}$. This projection is computed exactly as in \cite{bilevelTGV} by using the same method and parameters mentioned there. In our numerical computations we set $\underline{a}=10^{-8}$, $\overline{a}=5$, $n_{w}=7$, $\lambda=10^{-11}$, $\tau^{0}=10^{-3}$, $c=10^{-12}$, $\theta_{-}=0.25$, $\theta_{+}=2$. As initializations for $\alpha$, we use the constant functions $\alpha^{0}=0.5$ and $\alpha^{0}=1$ for the TV and TV$^{2}$ problems respectively. Regarding $\underline{\sigma}, \overline{\sigma}$, we use the formulas $\underline{\sigma}^{2}=\sigma^{2}(1-\frac{\sqrt{2}}{n_{w}})$ and $\overline{\sigma}^{2}=\sigma^{2}(1+\frac{\sqrt{2}}{n_{w}})$, which are based on the statistics of the extremes, see \cite[Section 4.2.1]{hintermuellerPartII}. In all our noisy images, the Gaussian noise has zero mean and variance $\sigma^{2}=0.01$. We terminate Algorithm \ref{alg:bilevelTVandTV2} after a fixed number of iterations $\mathrm{maxit}=100$, after which no noticeable reduction in the reduced objective function is observed, see also Figure \ref{fig:Fvalues}.

In order to produce a spatially varying Huber parameter $\gamma$ as described before, we solve an auxiliary bilevel Tikhonov problem where the lower level problem corresponds to $G_{T}(u)\coloneqq u-\mathrm{div}(\tilde{\alpha} \nabla u)-g=0$, i.e., the first order optimality condition of a variational denoising problem with $\frac{1}{2}\int_{\Omega}\tilde{\alpha} |\nabla u|^{2}\dif x$ as regularizer and $L^{2}$ fidelity term. In order to do so, we utilize again the  projected gradient algorithm described in Algorithm \ref{alg:bilevelTVandTV2}, adjusted to this regularizer. We use no additional $H^{1}$ regularization for $\gamma$ and the $H^{1}$-projection $P_{\mathcal{C}}$ is substituted by a simple $L^{2}$ projection, that is, $\alpha^{k+1}=\max (\min (\alpha^{k}-\tau^{k}\nabla_{a}\hat{\mathcal{F}}(\alpha^{k}), \overline{\alpha}), \underline{\alpha})$. We use 100 projected gradient iterations with $n_{w}$, $\underline{\sigma}, \overline{\sigma}$, $\tau^{0}$, $c$, $\theta_{-}$, $\theta_{+}$ as before, as well as $\underline{\alpha}=10^{-8}$, $\overline{\alpha}=15$, $\alpha_{0}=15$. The equation $G_{T}$ is solved exactly with a linear system solver. We use again 100 projected gradient iterations to get an output weight $\tilde{a}$. Then, as we mentioned in \eqref{inv_alphatilde} we define $\gamma=s/\tilde{\alpha}$. In all our experiments, we set $s=0.1$.

For comparison purposes, we also report TV and TV$^{2}$ denoising results with a scalar Huber parameter $\gamma$, which is always set $\gamma=10^{-3}$. We do that for both scalar and weighted regularization parameters $\alpha$. In the first case, we manually select the parameter $\alpha$ that maximizes the PSNR of the denoised image, computed with a semismooth Newton method as previously mentioned. The second case is computed exactly as in Algorithm \ref{alg:bilevelTVandTV2}. We also report the TGV reconstructions, both scalar and weighted versions, which are computed with the Chambolle-Pock primal-dual method \cite{chambolle2011first} as described in \cite{tgvcolour}. For the scalar case, again we manually select the TGV parameters $\alpha_{0}, \alpha_{1}$ that maximize the PSNR. For the weighted case, we use the spatially varying weights $\alpha_{0}, \alpha_{1}$ as produced in \cite{bilevelTGV} for the same image examples we are considering here.  In that work, these weights were computed via the ground truth-free bilevel approach but using a regularized lower level problem, i.e.\ with additional  $H_{1}$ regularizations in the primal variables of the TGV minimization problem. We remark that it turns out that when these weights are  directly fed into the Chambolle-Pock algorithm for the non-regularized problem as we do here, they produce a result of higher quality, hence the discrepancy between the PSNR and SSIM values we report here and then ones reported in \cite{bilevelTGV}.

 \begin{algorithm}[t!]
  \caption{\newline Projected gradient for  the bilevel huber TV (resp.\ TV$^{2}$) problems \eqref{bilevels_final1}--\eqref{bilevels_final2} \label{alg:bilevelTVandTV2}}
  \begin{algorithmic}
    \Statex {\textbf{Input}: $g$, $\underline{\alpha}$, $\overline{\alpha}$,  $\underline{\sigma}$, $\overline{\sigma}$, $\lambda$,  $\gamma$,  $n_{w}$, $\tau^{0}$,  $0<c<1$, $0<\theta_{-}<1\le \theta_{+}$}
    \Statex{\textbf{Initialize}: $\alpha^{0}\in\mathcal{C}$,  and set $k=0$.}
 \Repeat
     \State{Use Semismooth Newton to solve \eqref{l1_res1}--\eqref{l1_res2}  (resp.  \eqref{l2_res1}--\eqref{l2_res2}), i.e.\ 
          \[G_{1}(u^{k},p^{k})=0, \quad (\text{resp. } G_{2}(u^{k},p^{k})=0)\]}
    \State{Solve  for $(u^{\ast}, p^{\ast})$ the adjoint equation ($i=1$, resp.\ $i=2$)
    \[(D_{(u,p)} G_{i}(u^{k}, p^{k}))^{\top} (u^{\ast}, p^{\ast})=- D_{u} \mathcal{F} (u^{k},\alpha^{k})\]
    } 
  \State{Compute the derivative of the reduced objective w.r.t.\ $\alpha$ ($i=1$, resp.\ $i=2$)  as 
  \[\hat{\mathcal{F}}'(\alpha^{k})
=(D_{\alpha}G_{i}(\alpha^{k}))^{\top} (u^{\ast}, p^{\ast})+ D_{\alpha}\mathcal{F}(\alpha_{k})\] 
   } 
  \State{Compute the reduced gradient
  \begin{align*}
  \nabla_{\alpha} \hat{\mathcal{F}}(\alpha^{k})
  &=(I-\Delta_{N})^{-1}\hat{\mathcal{F}}'(\alpha^{k})
  \end{align*}
  }
\State{Compute the trial points 
\begin{align*}
\alpha^{k+1}&=P_{\mathcal{C}}\big(\alpha^{k}-\tau^{k}  \nabla_{\alpha} \hat{\mathcal{F}}(\alpha^{k})\big),\quad 
\end{align*}
}
\While{
\begin{align*}
\hat{\mathcal{F}}(\alpha^{k+1})
> \hat{\mathcal{F}}(\alpha^{k})
+c\big (\hat{\mathcal{F}}'(\alpha^{k})^{\top}(\alpha^{k+1}-\alpha^{k})\big)
\end{align*}
}{ (Armijo line search)}
\State{Set $\tau^{k}\coloneqq\theta_{-}\tau^{k}$,  and re-compute 
\begin{align*}
\alpha^{k+1}&=P_{\mathcal{C}}\big(\alpha^{k}-\tau^{k}  \nabla_{\alpha} \hat{\mathcal{F}}(\alpha^{k})\big)
\end{align*}
}
  \EndWhile
  \State{Update $\tau^{k+1}=\theta_{+}\tau^{k}$,  and $k\coloneqq k+1$}
 \Until{some stopping condition is satisfied}
  \end{algorithmic}
\end{algorithm}

\begin{figure}[t!]
	\centering		
	\begin{minipage}[t]{0.23\textwidth}
	\includegraphics[width=0.95\textwidth, trim={0cm 0cm 0cm 0cm},clip]{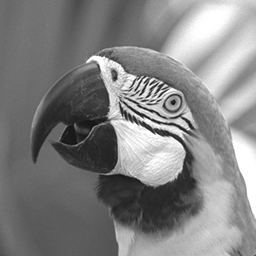}
	\end{minipage}
	\begin{minipage}[t]{0.23\textwidth}
	\includegraphics[width=0.95\textwidth, trim={0cm 0cm 0cm 0cm},clip]{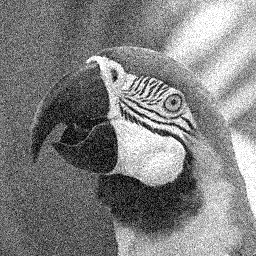}
	\end{minipage}
	\begin{minipage}[t]{0.26\textwidth}
	\includegraphics[width=0.95\textwidth, trim={2.0cm 3.5cm 2.0cm 2.9cm},clip]{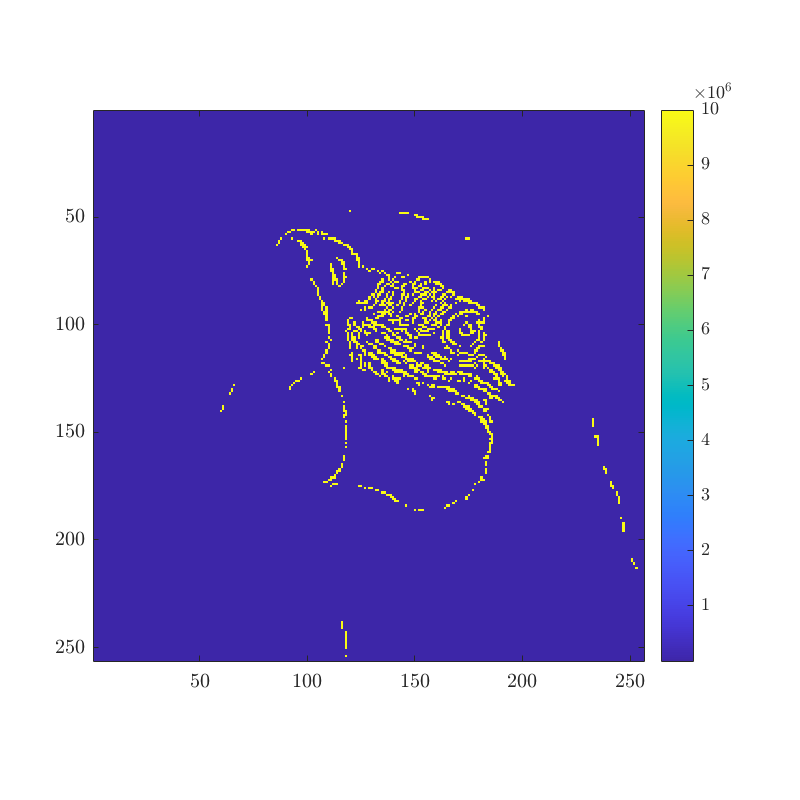}
	\end{minipage}
	
	\begin{minipage}[t]{0.23\textwidth}
	\centering \scalebox{.65}{Ground truth} \\[-3pt] \scalebox{.65}{PSNR=$\infty$, SSIM=1.000}\
	\end{minipage}
	\begin{minipage}[t]{0.23\textwidth}
	\centering \scalebox{.65}{Gaussian noise, $\sigma^{2}=0.01$} \\[-3pt] \scalebox{.65}{PSNR=20.04, SSIM=0.2773} 
	\end{minipage}
	\begin{minipage}[t]{0.23\textwidth}
	\centering \scalebox{.65}{Spatially varying $\gamma$} 
	\end{minipage}\vspace{1em}

	\begin{minipage}[t]{0.23\textwidth}
	\includegraphics[width=0.95\textwidth, trim={0cm 0cm 0cm 0cm},clip]{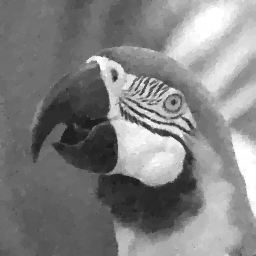}
	\end{minipage}
	\begin{minipage}[t]{0.23\textwidth}
	\includegraphics[width=0.95\textwidth, trim={0cm 0cm 0cm 0cm},clip]{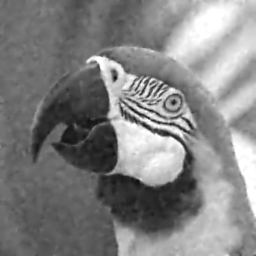}
	\end{minipage}
	\begin{minipage}[t]{0.23\textwidth}
	\includegraphics[width=0.95\textwidth, trim={0cm 0cm 0cm 0cm},clip]{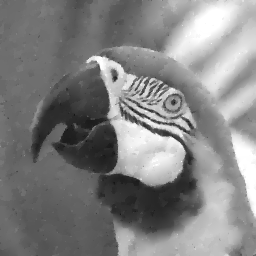}
	\end{minipage}
	\begin{minipage}[t]{0.23\textwidth}
	\includegraphics[width=0.95\textwidth, trim={0cm 0cm 0cm 0cm},clip]{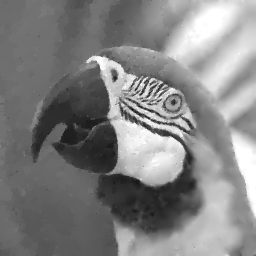}
	\end{minipage}

	\begin{minipage}[t]{0.23\textwidth}
	 \centering \scalebox{.65}{scalar Huber TV} \\[-3pt] \scalebox{.65}{PSNR=29.25, SSIM=0.8354}
	\end{minipage}
	\begin{minipage}[t]{0.23\textwidth}
	 \centering \scalebox{.65}{scalar Huber TV$^{2}$} \\[-3pt] \scalebox{.65}{PSNR=29.28, SSIM=0.8305}
	\end{minipage}
	\begin{minipage}[t]{0.23\textwidth}
	 \centering \scalebox{.65}{scalar TGV} \\[-3pt] \scalebox{.65}{PSNR=29.50, SSIM=0.8509}
	\end{minipage}
	\begin{minipage}[t]{0.23\textwidth}
	\centering \scalebox{.65}{Bilevel weighted TGV}\\[-3pt] \scalebox{.65}{PSNR=\textbf{29.84}, SSIM=0.8606} 
	\end{minipage}\vspace{1em}
	
       \begin{minipage}[t]{0.23\textwidth}
	\includegraphics[width=0.95\textwidth, trim={0cm 0cm 0cm 0cm},clip]{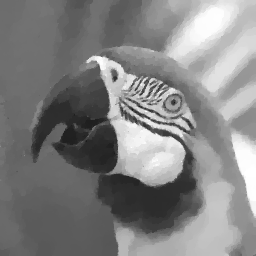}
	\end{minipage}	
	  \begin{minipage}[t]{0.23\textwidth}
	\includegraphics[width=0.95\textwidth, trim={0cm 0cm 0cm 0cm},clip]{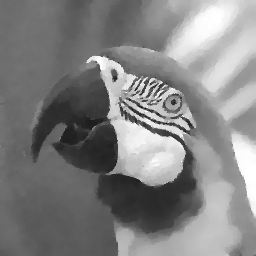}
	\end{minipage}
	  \begin{minipage}[t]{0.23\textwidth}
	\includegraphics[width=0.95\textwidth, trim={0cm 0cm 0cm 0cm},clip]{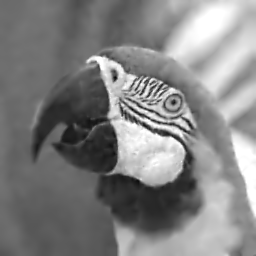}
	\end{minipage}
	\begin{minipage}[t]{0.23\textwidth}
	\includegraphics[width=0.95\textwidth, trim={0cm 0cm 0cm 0cm},clip]{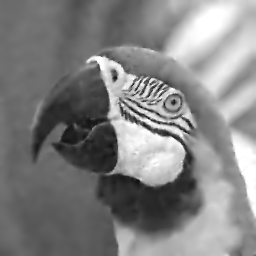}
	\end{minipage}
	
	\begin{minipage}[t]{0.23\textwidth}
	\centering \scalebox{.65}{Bilevel weighted Huber TV}\\[-3pt] \scalebox{.65}{with scalar  $\gamma$}\\ \scalebox{.65}{PSNR=29.20,  SSIM=0.8549 }
	\end{minipage}
	\begin{minipage}[t]{0.23\textwidth}
	\centering \scalebox{.65}{Bilevel weighted Huber TV}\\[-3pt] \scalebox{.65}{with spatially varying  $\gamma$}\\ \scalebox{.65}{PSNR=28.92,  SSIM=0.8571 }
	\end{minipage}
    \begin{minipage}[t]{0.23\textwidth}
		\centering \scalebox{.65}{Bilevel weighted Huber TV$^{2}$}\\[-3pt] \scalebox{.65}{with scalar  $\gamma$}\\ \scalebox{.65}{PSNR=29.81,  SSIM=\textbf{0.8705} }
	\end{minipage}
	    \begin{minipage}[t]{0.23\textwidth}
		\centering \scalebox{.65}{Bilevel weighted Huber TV$^{2}$}\\[-3pt] \scalebox{.65}{with spatially varying  $\gamma$}\\ \scalebox{.65}{PSNR=\textbf{29.84},  SSIM=0.8700 }
	\end{minipage}

       \begin{minipage}[t]{0.23\textwidth}
	\includegraphics[width=0.95\textwidth, trim={0cm 3.5cm 0cm 0cm},clip]{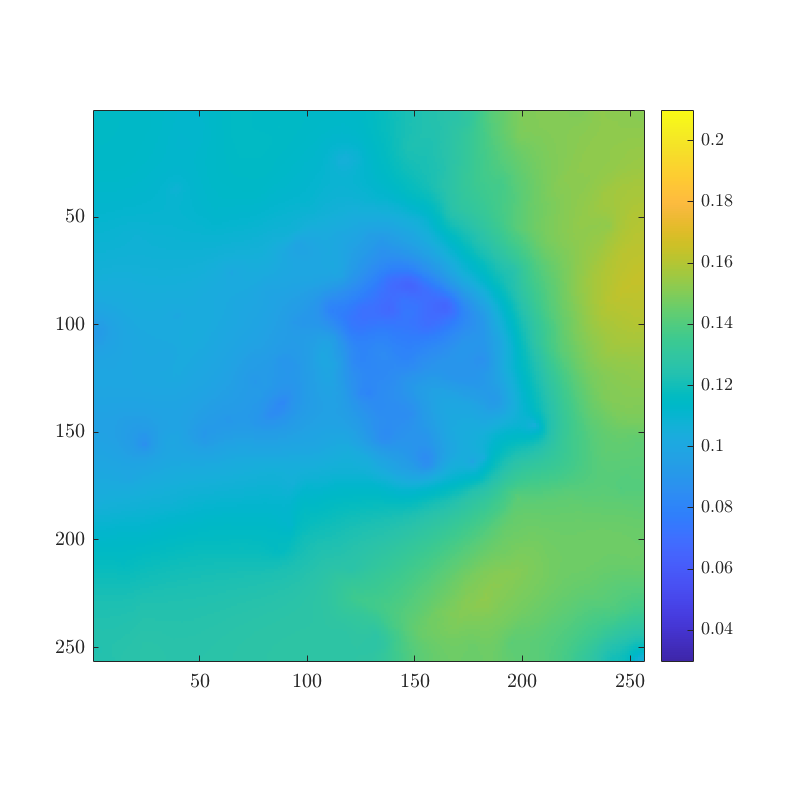}
	\end{minipage}	
	  \begin{minipage}[t]{0.23\textwidth}
	\includegraphics[width=0.95\textwidth, trim={0cm 3.5cm 0cm 0cm},clip]{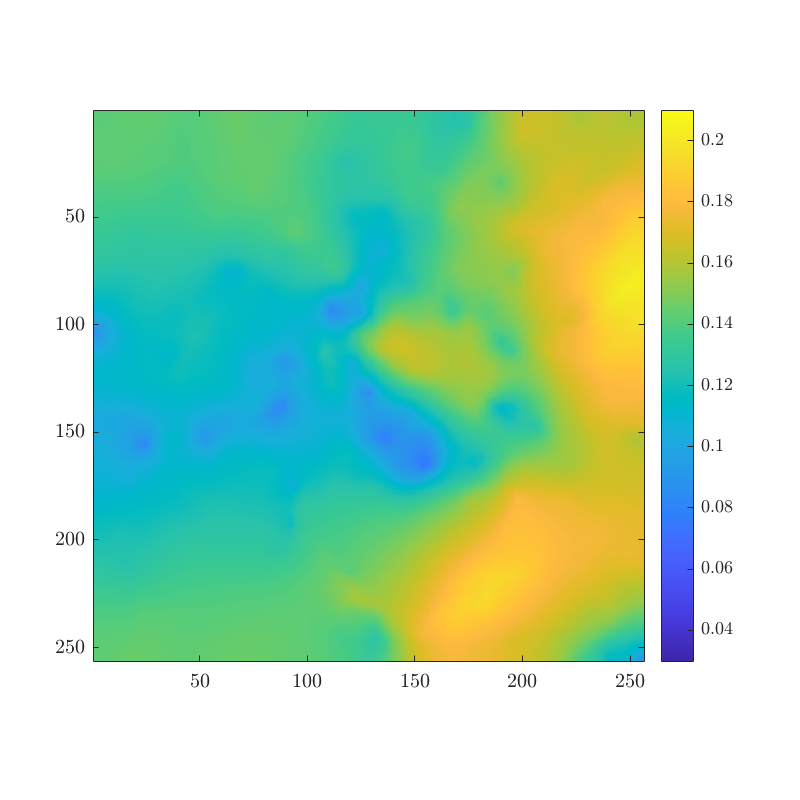}
	\end{minipage}
	  \begin{minipage}[t]{0.23\textwidth}
	\includegraphics[width=0.95\textwidth, trim={0cm 3.5cm 0cm 0cm},clip]{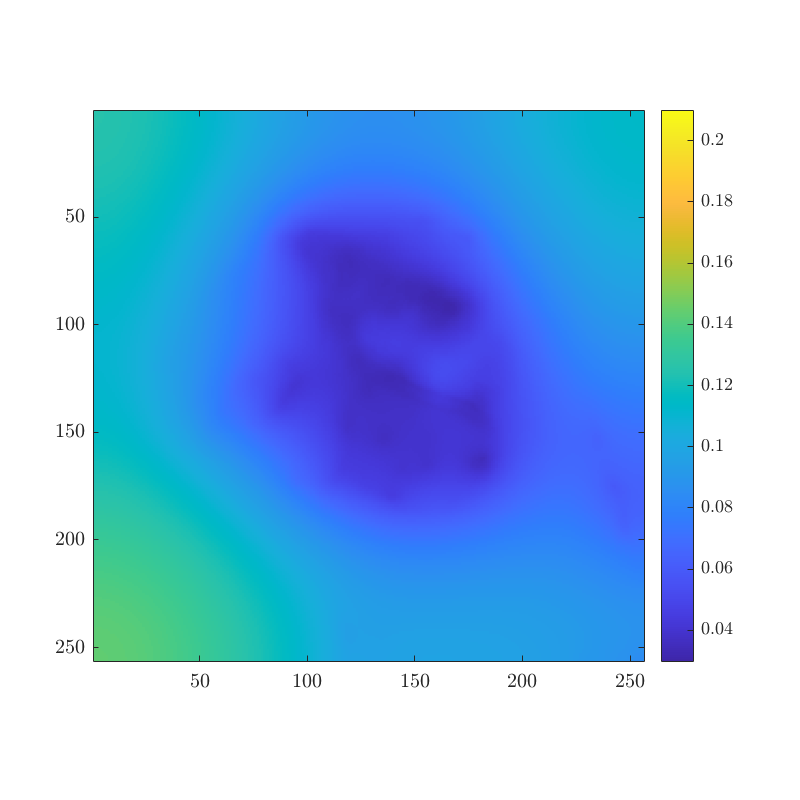}
	\end{minipage}
	\begin{minipage}[t]{0.23\textwidth}
	\includegraphics[width=0.95\textwidth, trim={0cm 3.5cm 0cm 0cm},clip]{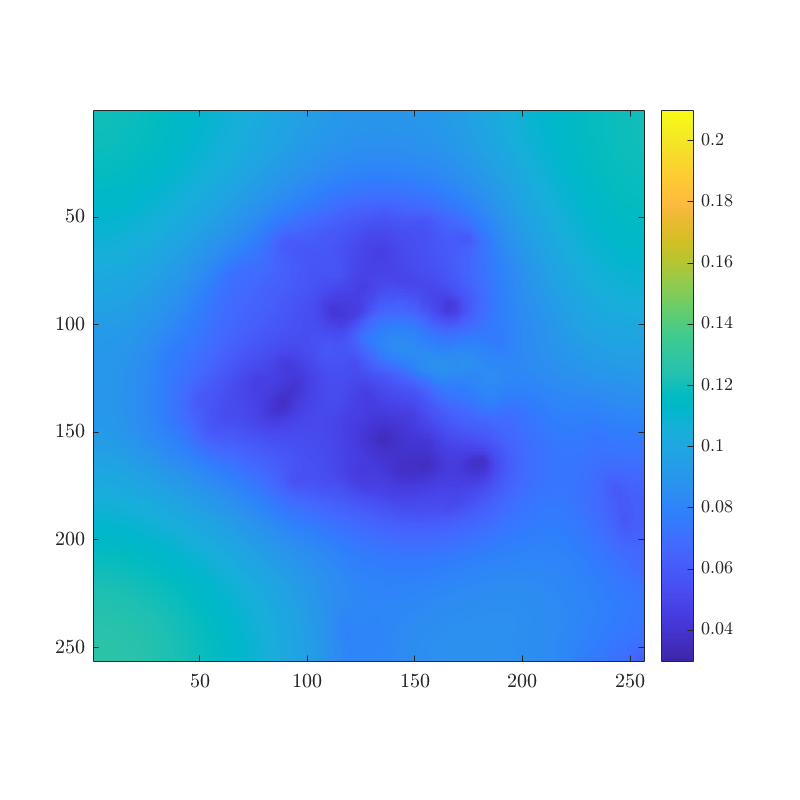}
	\end{minipage}
	
	\begin{minipage}[t]{0.23\textwidth}
	\centering \scalebox{.65}{Weight $\alpha$ of Huber TV}
	\end{minipage}
	\begin{minipage}[t]{0.23\textwidth}
\centering \scalebox{.65}{Weight $\alpha$ of Huber TV}
	\end{minipage}
    \begin{minipage}[t]{0.23\textwidth}
\centering \scalebox{.65}{Weight $\alpha$ of Huber TV$^{2}$}
	\end{minipage}
	    \begin{minipage}[t]{0.23\textwidth}
\centering \scalebox{.65}{Weight $\alpha$ of Huber TV$^{2}$}
	\end{minipage}
	\caption{\emph{Parrot} image: Weighted Huber TV and TV$^{2}$ denoising with spatially varying Huber parameter $\gamma$ and regularization parameter $\alpha$. The weights $\alpha$ are produced with the ground truth-free bilevel approach using $\mathcal{F}_{\stat}$. The highest PSNR and SSIM values are highlighted in bold font.}
		\label{fig:parrot_stat}
\end{figure}

\begin{figure}[t!]
	\centering		
	\begin{minipage}[t]{0.23\textwidth}
	\includegraphics[width=0.95\textwidth, trim={4cm 4cm 2cm 2.5cm},clip]{figures/parrot_clean.png}
	\end{minipage}
	\begin{minipage}[t]{0.23\textwidth}
	\includegraphics[width=0.95\textwidth, trim={4cm 4cm 2cm 2.5cm},clip]{figures/parrot_noisy_01.png}
	\end{minipage}
	
	\begin{minipage}[t]{0.23\textwidth}
	\centering \scalebox{.65}{Ground truth} \\[-3pt] \scalebox{.65}{PSNR=$\infty$, SSIM=1.000}\
	\end{minipage}
	\begin{minipage}[t]{0.23\textwidth}
	\centering \scalebox{.65}{Gaussian noise, $\sigma^{2}=0.01$} \\[-3pt] \scalebox{.65}{PSNR=20.04, SSIM=0.2773} 
	\end{minipage}\vspace{1em}

	\begin{minipage}[t]{0.23\textwidth}
	\includegraphics[width=0.95\textwidth, trim={4cm 4cm 2cm 2.5cm},clip]{paper_results/hubertv_bestscalar_psnr_u}
	\end{minipage}
	\begin{minipage}[t]{0.23\textwidth}
	\includegraphics[width=0.95\textwidth, trim={4cm 4cm 2cm 2.5cm},clip]{paper_results/hubertv2_bestscalar_psnr_u}
	\end{minipage}
	\begin{minipage}[t]{0.23\textwidth}
	\includegraphics[width=0.95\textwidth, trim={4cm 4cm 2cm 2.5cm},clip]{paper_results/tgv_bestscalar_psnr_u}
	\end{minipage}
	\begin{minipage}[t]{0.23\textwidth}
	\includegraphics[width=0.95\textwidth, trim={4cm 4cm 2cm 2.5cm},clip]{paper_results/tgv_spatial_u}
	\end{minipage}

	\begin{minipage}[t]{0.23\textwidth}
	 \centering \scalebox{.65}{scalar Huber TV} \\[-3pt] \scalebox{.65}{PSNR=29.25, SSIM=0.8354}
	\end{minipage}
	\begin{minipage}[t]{0.23\textwidth}
	 \centering \scalebox{.65}{scalar Huber TV$^{2}$} \\[-3pt] \scalebox{.65}{PSNR=29.28, SSIM=0.8305}
	\end{minipage}
	\begin{minipage}[t]{0.23\textwidth}
	 \centering \scalebox{.65}{scalar TGV} \\[-3pt] \scalebox{.65}{PSNR=29.50, SSIM=0.8509}
	\end{minipage}
	\begin{minipage}[t]{0.23\textwidth}
	\centering \scalebox{.65}{Bilevel weighted TGV}\\[-3pt] \scalebox{.65}{PSNR=\textbf{29.84}, SSIM=0.8606} 
	\end{minipage}\vspace{1em}
	
       \begin{minipage}[t]{0.23\textwidth}
	\includegraphics[width=0.95\textwidth, trim={4cm 4cm 2cm 2.5cm},clip]{paper_results/hubertv_gamma_scalar_u.png}
	\end{minipage}	
	  \begin{minipage}[t]{0.23\textwidth}
	\includegraphics[width=0.95\textwidth, trim={4cm 4cm 2cm 2.5cm},clip]{paper_results/hubertv_gamma_spatial_01_u.png}
	\end{minipage}
	  \begin{minipage}[t]{0.23\textwidth}
	\includegraphics[width=0.95\textwidth, trim={4cm 4cm 2cm 2.5cm},clip]{paper_results/hubertv2_gamma_scalar_a0_1_u.png}
	\end{minipage}
	\begin{minipage}[t]{0.23\textwidth}
	\includegraphics[width=0.95\textwidth, trim={4cm 4cm 2cm 2.5cm},clip]{paper_results/hubertv2_gamma_spatial_a0_1_u.png}
	\end{minipage}
	
	\begin{minipage}[t]{0.23\textwidth}
	\centering \scalebox{.65}{Bilevel weighted Huber TV}\\[-3pt] \scalebox{.65}{with scalar  $\gamma$}\\ \scalebox{.65}{PSNR=29.20,  SSIM=0.8549 }
	\end{minipage}
	\begin{minipage}[t]{0.23\textwidth}
	\centering \scalebox{.65}{Bilevel weighted Huber TV}\\[-3pt] \scalebox{.65}{with spatially varying  $\gamma$}\\ \scalebox{.65}{PSNR=28.92,  SSIM=0.8571 }
	\end{minipage}
    \begin{minipage}[t]{0.23\textwidth}
		\centering \scalebox{.65}{Bilevel weighted Huber TV$^{2}$}\\[-3pt] \scalebox{.65}{with scalar  $\gamma$}\\ \scalebox{.65}{PSNR=29.81,  SSIM=\textbf{0.8705} }
	\end{minipage}
	    \begin{minipage}[t]{0.23\textwidth}
		\centering \scalebox{.65}{Bilevel weighted Huber TV$^{2}$}\\[-3pt] \scalebox{.65}{with spatially varying  $\gamma$}\\ \scalebox{.65}{PSNR=\textbf{29.84},  SSIM=0.8700 }
	\end{minipage}
	\caption{Details of images shown in  Figure \ref{fig:parrot_stat}}
	\label{fig:parrot_stat_details}
\end{figure}

\begin{figure}[h!]
	\centering		
	\begin{minipage}[t]{0.23\textwidth}
	\includegraphics[width=0.95\textwidth, trim={0cm 0cm 0cm 0cm},clip]{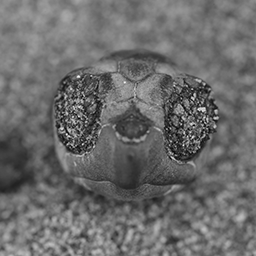}
	\end{minipage}
	\begin{minipage}[t]{0.23\textwidth}
	\includegraphics[width=0.95\textwidth, trim={0cm 0cm 0cm 0cm},clip]{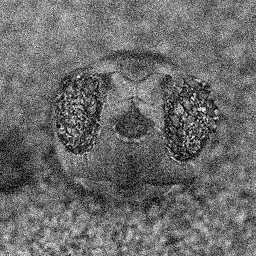}
	\end{minipage}
	\begin{minipage}[t]{0.26\textwidth}
	\includegraphics[width=0.95\textwidth, trim={2.0cm 3.5cm 2.0cm 2.9cm},clip]{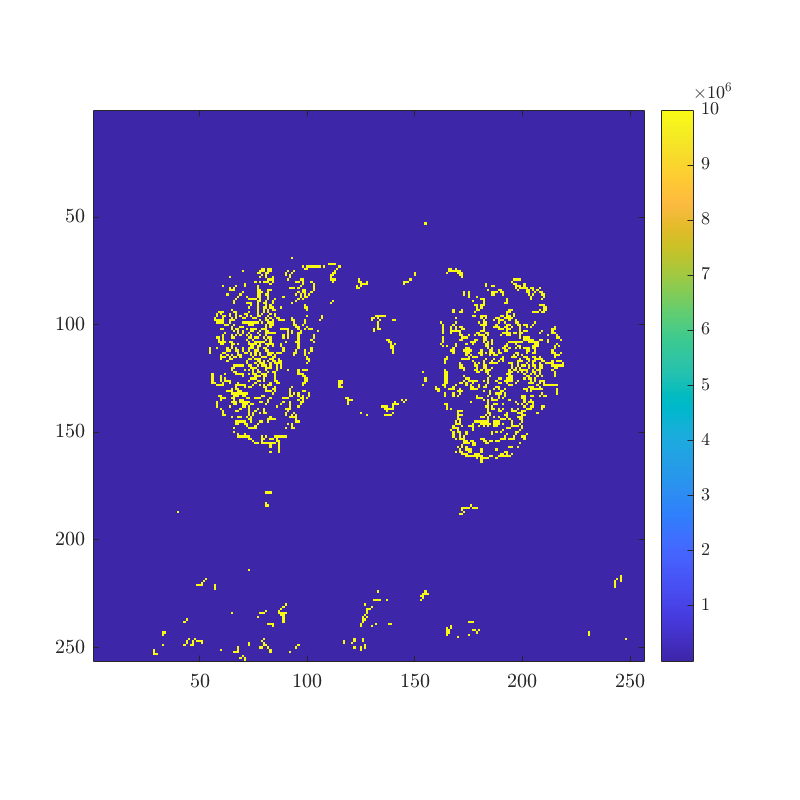}
	\end{minipage}
	
	\begin{minipage}[t]{0.23\textwidth}
	\centering \scalebox{.65}{Ground truth} \\[-3pt] \scalebox{.65}{PSNR=$\infty$, SSIM=1.000}\
	\end{minipage}
	\begin{minipage}[t]{0.23\textwidth}
	\centering \scalebox{.65}{Gaussian noise, $\sigma^{2}=0.01$} \\[-3pt] \scalebox{.65}{PSNR=20.00, SSIM=0.3349} 
	\end{minipage}
	\begin{minipage}[t]{0.23\textwidth}
	\centering \scalebox{.65}{Spatially varying $\gamma$} 
	\end{minipage}\vspace{1em}

	\begin{minipage}[t]{0.23\textwidth}
	\includegraphics[width=0.95\textwidth, trim={0cm 0cm 0cm 0cm},clip]{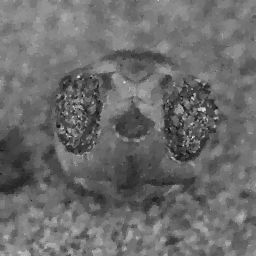}
	\end{minipage}
	\begin{minipage}[t]{0.23\textwidth}
	\includegraphics[width=0.95\textwidth, trim={0cm 0cm 0cm 0cm},clip]{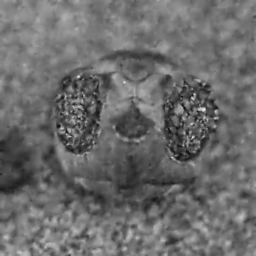}
	\end{minipage}
	\begin{minipage}[t]{0.23\textwidth}
	\includegraphics[width=0.95\textwidth, trim={0cm 0cm 0cm 0cm},clip]{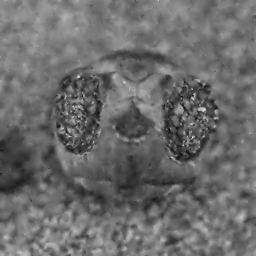}
	\end{minipage}
	\begin{minipage}[t]{0.23\textwidth}
	\includegraphics[width=0.95\textwidth, trim={0cm 0cm 0cm 0cm},clip]{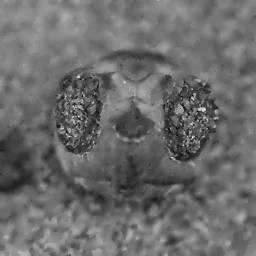}
	\end{minipage}

	\begin{minipage}[t]{0.23\textwidth}
	 \centering \scalebox{.65}{scalar Huber TV} \\[-3pt] \scalebox{.65}{PSNR=27.75, SSIM=0.7701}
	\end{minipage}
	\begin{minipage}[t]{0.23\textwidth}
	 \centering \scalebox{.65}{scalar Huber TV$^{2}$} \\[-3pt] \scalebox{.65}{PSNR=28.22, SSIM=0.8142}
	\end{minipage}
	\begin{minipage}[t]{0.23\textwidth}
	 \centering \scalebox{.65}{scalar TGV} \\[-3pt] \scalebox{.65}{PSNR=28.20, SSIM=0.8132}
	\end{minipage}
	\begin{minipage}[t]{0.23\textwidth}
	\centering \scalebox{.65}{Bilevel weighted TGV}\\[-3pt] \scalebox{.65}{PSNR=28.33, SSIM=0.8145} 
	\end{minipage}\vspace{1em}
	
       \begin{minipage}[t]{0.23\textwidth}
	\includegraphics[width=0.95\textwidth, trim={0cm 0cm 0cm 0cm},clip]{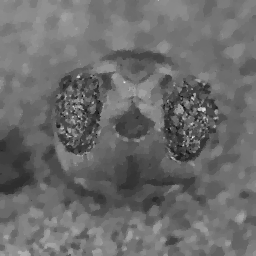}
	\end{minipage}	
	  \begin{minipage}[t]{0.23\textwidth}
	\includegraphics[width=0.95\textwidth, trim={0cm 0cm 0cm 0cm},clip]{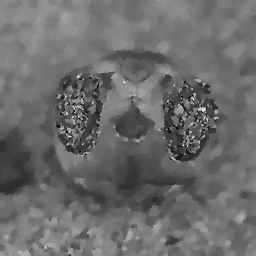}
	\end{minipage}
	  \begin{minipage}[t]{0.23\textwidth}
	\includegraphics[width=0.95\textwidth, trim={0cm 0cm 0cm 0cm},clip]{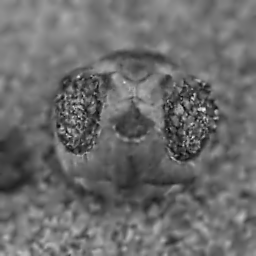}
	\end{minipage}
	\begin{minipage}[t]{0.23\textwidth}
	\includegraphics[width=0.95\textwidth, trim={0cm 0cm 0cm 0cm},clip]{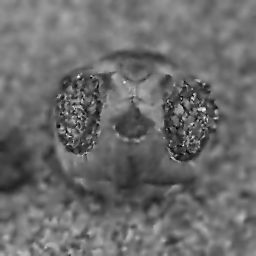}
	\end{minipage}
	
	\begin{minipage}[t]{0.23\textwidth}
	\centering \scalebox{.65}{Bilevel weighted Huber TV}\\[-3pt] \scalebox{.65}{with scalar  $\gamma$}\\ \scalebox{.65}{PSNR=27.50,  SSIM=0.7702 }
	\end{minipage}
	\begin{minipage}[t]{0.23\textwidth}
	\centering \scalebox{.65}{Bilevel weighted Huber TV}\\[-3pt] \scalebox{.65}{with spatially varying  $\gamma$}\\ \scalebox{.65}{PSNR=27.15,  SSIM=0.7688 }
	\end{minipage}
    \begin{minipage}[t]{0.23\textwidth}
		\centering \scalebox{.65}{Bilevel weighted Huber TV$^{2}$}\\[-3pt] \scalebox{.65}{with scalar  $\gamma$}\\ \scalebox{.65}{PSNR=\textbf{28.66},  SSIM=\textbf{0.8367} }
	\end{minipage}
	    \begin{minipage}[t]{0.23\textwidth}
		\centering \scalebox{.65}{Bilevel weighted Huber TV$^{2}$}\\[-3pt] \scalebox{.65}{with spatially varying  $\gamma$}\\ \scalebox{.65}{PSNR=28.44,  SSIM=0.8285 }
	\end{minipage}

       \begin{minipage}[t]{0.23\textwidth}
	\includegraphics[width=0.95\textwidth, trim={0cm 3.5cm 0cm 0cm},clip]{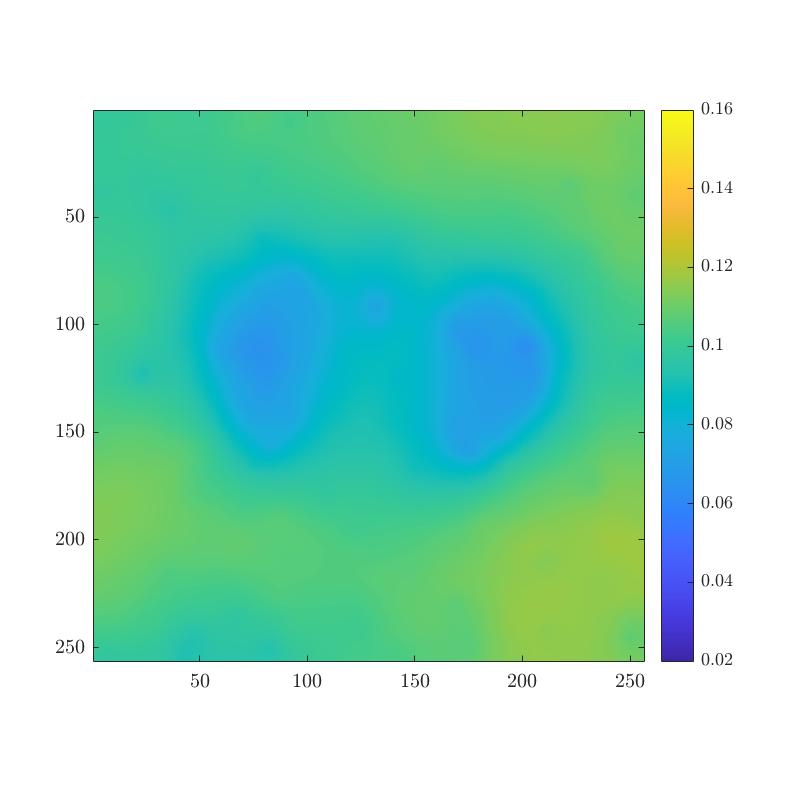}
	\end{minipage}	
	  \begin{minipage}[t]{0.23\textwidth}
	\includegraphics[width=0.95\textwidth, trim={0cm 3.5cm 0cm 0cm},clip]{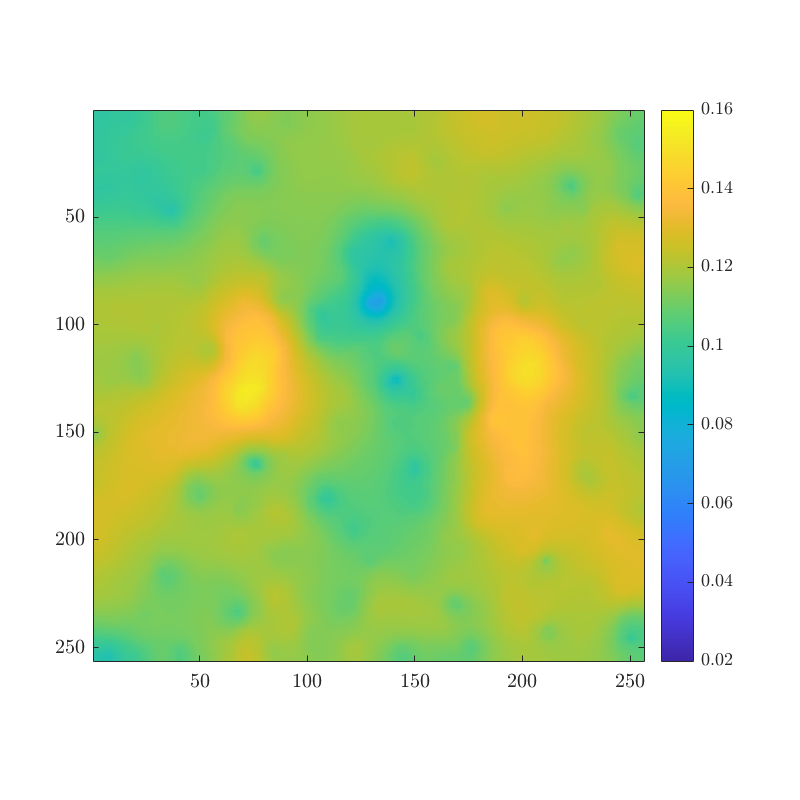}
	\end{minipage}
	  \begin{minipage}[t]{0.23\textwidth}
	\includegraphics[width=0.95\textwidth, trim={0cm 3.5cm 0cm 0cm},clip]{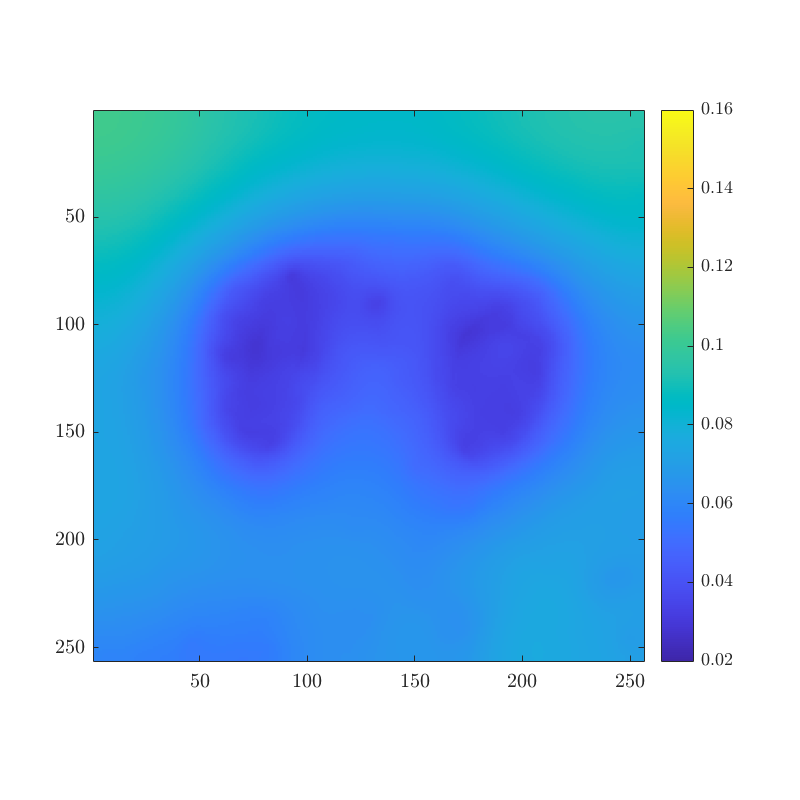}
	\end{minipage}
	\begin{minipage}[t]{0.23\textwidth}
	\includegraphics[width=0.95\textwidth, trim={0cm 3.5cm 0cm 0cm},clip]{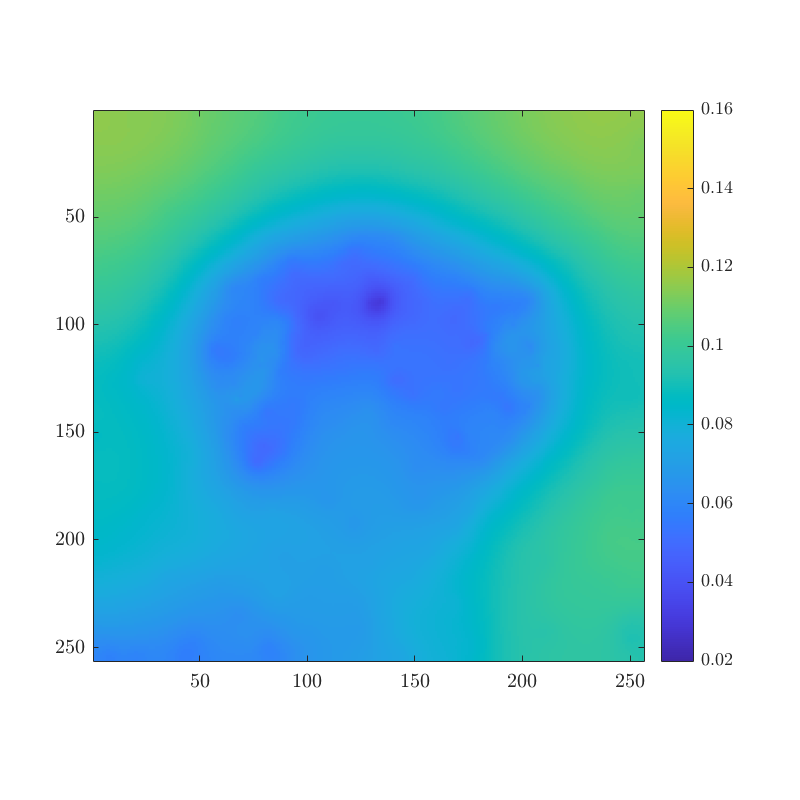}
	\end{minipage}
	
	\begin{minipage}[t]{0.23\textwidth}
	\centering \scalebox{.65}{Weight $\alpha$ of Huber TV}
	\end{minipage}
	\begin{minipage}[t]{0.23\textwidth}
\centering \scalebox{.65}{Weight $\alpha$ of Huber TV}
	\end{minipage}
    \begin{minipage}[t]{0.23\textwidth}
\centering \scalebox{.65}{Weight $\alpha$ of Huber TV$^{2}$}
	\end{minipage}
	    \begin{minipage}[t]{0.23\textwidth}
\centering \scalebox{.65}{Weight $\alpha$ of Huber TV$^{2}$}
	\end{minipage}

	\caption{\emph{Hatchling} image: Weighted Huber TV and TV$^{2}$ denoising with spatially varying Huber parameter $\gamma$ and regularization parameter $\alpha$. The weights $\alpha$ are produced with the ground truth-free bilevel approach using $\mathcal{F}_{\stat}$. The highest PSNR and SSIM values are highlighted in bold font.}
		\label{fig:hatchling_stat}
\end{figure}

\begin{figure}[h!]
	\centering		
	\begin{minipage}[t]{0.23\textwidth}
	\includegraphics[width=0.95\textwidth, trim={4.5cm 3.2cm 0.5cm 2.3cm},clip]{figures/hatchling_clean.png}
	\end{minipage}
	\begin{minipage}[t]{0.23\textwidth}
	\includegraphics[width=0.95\textwidth, trim={4.5cm 3.2cm 0.5cm 2.3cm},clip]{figures/hatchling_noisy_01.png}
	\end{minipage}
	
	\begin{minipage}[t]{0.23\textwidth}
	\centering \scalebox{.65}{Ground truth} \\[-3pt] \scalebox{.65}{PSNR=$\infty$, SSIM=1.000}\
	\end{minipage}
	\begin{minipage}[t]{0.23\textwidth}
	\centering \scalebox{.65}{Gaussian noise, $\sigma^{2}=0.01$} \\[-3pt] \scalebox{.65}{PSNR=20.04, SSIM=0.2773} 
	\end{minipage}\vspace{1em}

	\begin{minipage}[t]{0.23\textwidth}
	\includegraphics[width=0.95\textwidth, trim={4.5cm 3.2cm 0.5cm 2.3cm},clip]{paper_results/h_hubertv_bestscalar_psnr_u}
	\end{minipage}
	\begin{minipage}[t]{0.23\textwidth}
	\includegraphics[width=0.95\textwidth, trim={4.5cm 3.2cm 0.5cm 2.3cm},clip]{paper_results/h_hubertv2_bestscalar_psnr_u}
	\end{minipage}
	\begin{minipage}[t]{0.23\textwidth}
	\includegraphics[width=0.95\textwidth, trim={4.5cm 3.2cm 0.5cm 2.3cm},clip]{paper_results/h_tgv_bestscalar_psnr_u}
	\end{minipage}
	\begin{minipage}[t]{0.23\textwidth}
	\includegraphics[width=0.95\textwidth, trim={4.5cm 3.2cm 0.5cm 2.3cm},clip]{paper_results/h_tgv_spatial_u}
	\end{minipage}

	\begin{minipage}[t]{0.23\textwidth}
	 \centering \scalebox{.65}{scalar Huber TV} \\[-3pt] \scalebox{.65}{PSNR=27.75, SSIM=0.7701}
	\end{minipage}
	\begin{minipage}[t]{0.23\textwidth}
	 \centering \scalebox{.65}{scalar Huber TV$^{2}$} \\[-3pt] \scalebox{.65}{PSNR=28.22, SSIM=0.8142}
	\end{minipage}
	\begin{minipage}[t]{0.23\textwidth}
	 \centering \scalebox{.65}{scalar TGV} \\[-3pt] \scalebox{.65}{PSNR=28.20, SSIM=0.8132}
	\end{minipage}
	\begin{minipage}[t]{0.23\textwidth}
	\centering \scalebox{.65}{Bilevel weighted TGV}\\[-3pt] \scalebox{.65}{PSNR=28.33, SSIM=0.8145} 
	\end{minipage}\vspace{1em}
	
       \begin{minipage}[t]{0.23\textwidth}
	\includegraphics[width=0.95\textwidth, trim={4.5cm 3.2cm 0.5cm 2.3cm},clip]{paper_results/h_hubertv_gamma_scalar_u.png}
	\end{minipage}	
	  \begin{minipage}[t]{0.23\textwidth}
	\includegraphics[width=0.95\textwidth, trim={4.5cm 3.2cm 0.5cm 2.3cm},clip]{paper_results/h_hubertv_gamma_spatial_01_u.png}
	\end{minipage}
	  \begin{minipage}[t]{0.23\textwidth}
	\includegraphics[width=0.95\textwidth, trim={4.5cm 3.2cm 0.5cm 2.3cm},clip]{paper_results/h_hubertv2_gamma_scalar_a0_1_u.png}
	\end{minipage}
	\begin{minipage}[t]{0.23\textwidth}
	\includegraphics[width=0.95\textwidth, trim={4.5cm 3.2cm 0.5cm 2.3cm},clip]{paper_results/h_hubertv2_gamma_spatial_a0_1_u.png}
	\end{minipage}
	
	\begin{minipage}[t]{0.23\textwidth}
	\centering \scalebox{.65}{Bilevel weighted Huber TV}\\[-3pt] \scalebox{.65}{with scalar  $\gamma$}\\ \scalebox{.65}{PSNR=27.50,  SSIM=0.7702 }
	\end{minipage}
	\begin{minipage}[t]{0.23\textwidth}
	\centering \scalebox{.65}{Bilevel weighted Huber TV}\\[-3pt] \scalebox{.65}{with spatially varying  $\gamma$}\\ \scalebox{.65}{PSNR=27.15,  SSIM=0.7688 }
	\end{minipage}
    \begin{minipage}[t]{0.23\textwidth}
		\centering \scalebox{.65}{Bilevel weighted Huber TV$^{2}$}\\[-3pt] \scalebox{.65}{with scalar  $\gamma$}\\ \scalebox{.65}{PSNR=\textbf{28.66},  SSIM=\textbf{0.8367} }
	\end{minipage}
	    \begin{minipage}[t]{0.23\textwidth}
		\centering \scalebox{.65}{Bilevel weighted Huber TV$^{2}$}\\[-3pt] \scalebox{.65}{with spatially varying  $\gamma$}\\ \scalebox{.65}{PSNR=28.44,  SSIM=0.8285 }
	\end{minipage}

	\caption{Details of images shown in  Figure \ref{fig:hatchling_stat}}
		\label{fig:hatchling_stat_details}
	
\end{figure}

\begin{figure}[h!]
	\centering		
	\begin{minipage}[t]{0.23\textwidth}
	\includegraphics[width=0.95\textwidth, trim={0cm 0cm 0cm 0cm},clip]{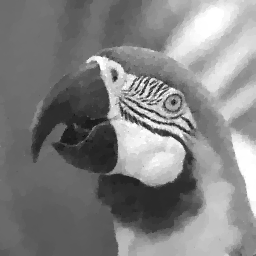}
	\end{minipage}
	\begin{minipage}[t]{0.23\textwidth}
	\includegraphics[width=0.95\textwidth, trim={0cm 0cm 0cm 0cm},clip]{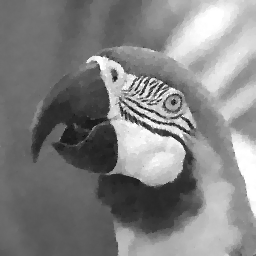}
	\end{minipage}
	\begin{minipage}[t]{0.23\textwidth}
	\includegraphics[width=0.95\textwidth, trim={0cm 0cm 0cm 0cm},clip]{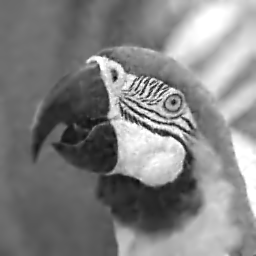}
	\end{minipage}
	\begin{minipage}[t]{0.23\textwidth}
	\includegraphics[width=0.95\textwidth, trim={0cm 0cm 0cm 0cm},clip]{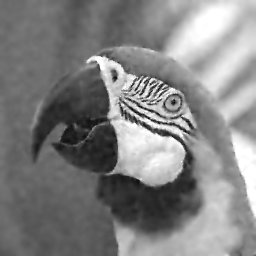}
	\end{minipage}

	\begin{minipage}[t]{0.23\textwidth}
	 \centering \scalebox{.65}{Bilevel weighted Huber TV} \\[-3pt] \scalebox{.65}{with scalar  $\gamma$} \\ \scalebox{.65}{PSNR=29.95, SSIM=0.8610}
	\end{minipage}
	\begin{minipage}[t]{0.23\textwidth}
	 \centering \scalebox{.65}{Bilevel weighted Huber TV} \\[-3pt] \scalebox{.65}{with spatially varying  $\gamma$} \\ \scalebox{.65}{PSNR=29.66, SSIM=0.8644}
	\end{minipage}
	\begin{minipage}[t]{0.23\textwidth}
	 \centering \scalebox{.65}{Bilevel weighted Huber TV$^{2}$}  \\[-3pt] \scalebox{.65}{with scalar  $\gamma$}  \\ \scalebox{.65}{PSNR=\textbf{30.27}, SSIM=\textbf{0.8736}}
	\end{minipage}
	\begin{minipage}[t]{0.23\textwidth}
	 \centering \scalebox{.65}{Bilevel weighted Huber  TV$^{2}$}  \\[-3pt] \scalebox{.65}{with spatially varying  $\gamma$}  \\ \scalebox{.65}{PSNR=30.19, SSIM=0.8651}
	\end{minipage}

	\begin{minipage}[t]{0.23\textwidth}
	\includegraphics[width=0.95\textwidth, trim={0cm 3.5cm 0cm 0cm},clip]{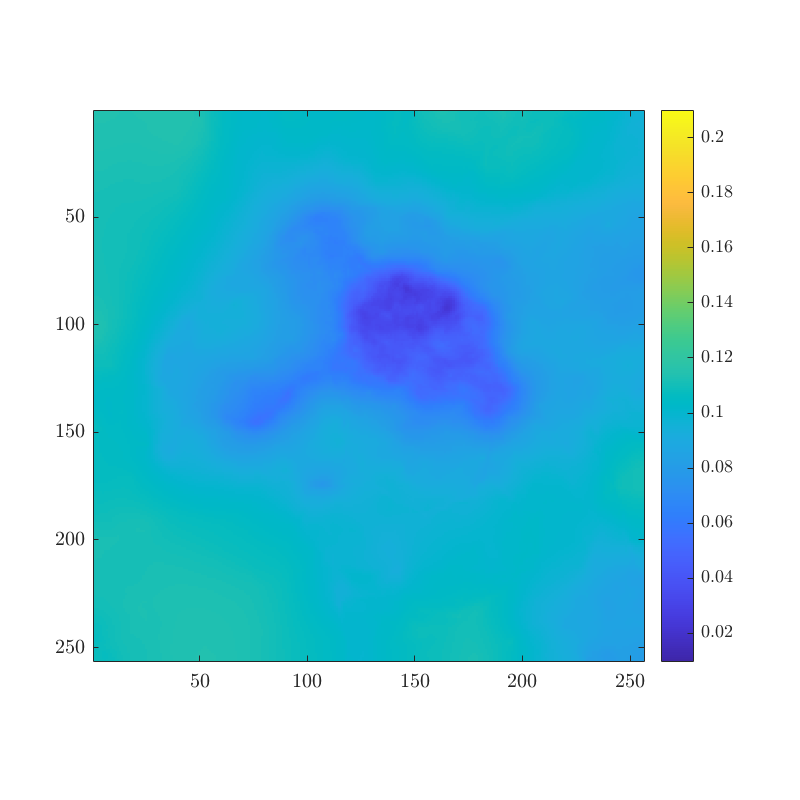}
	\end{minipage}	
	\begin{minipage}[t]{0.23\textwidth}
	\includegraphics[width=0.95\textwidth, trim={0cm 3.5cm 0cm 0cm},clip]{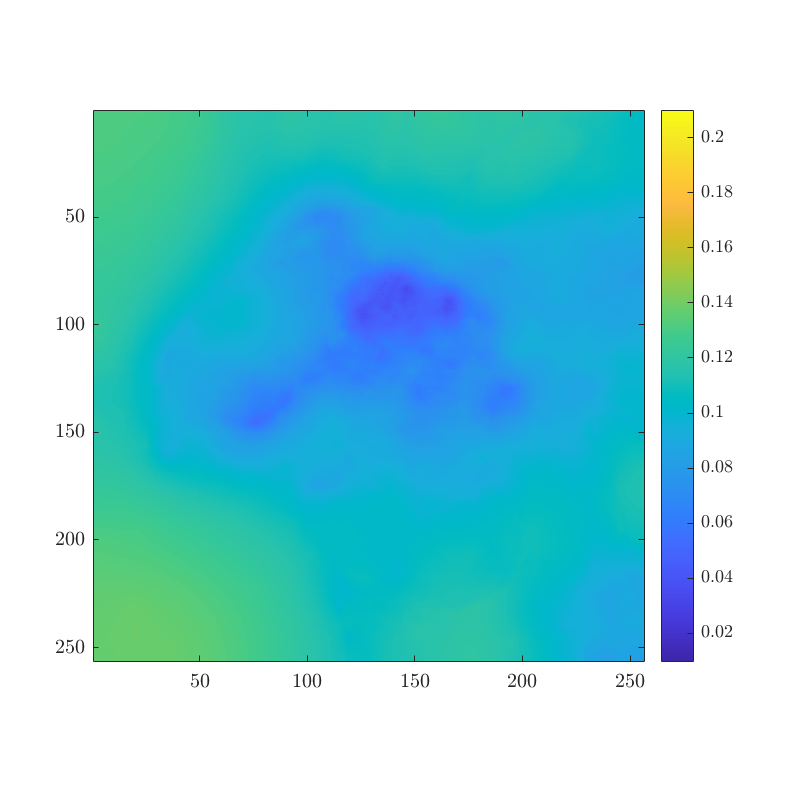}
	\end{minipage}		
	  \begin{minipage}[t]{0.23\textwidth}
	\includegraphics[width=0.95\textwidth, trim={0cm 3.5cm 0cm 0cm},clip]{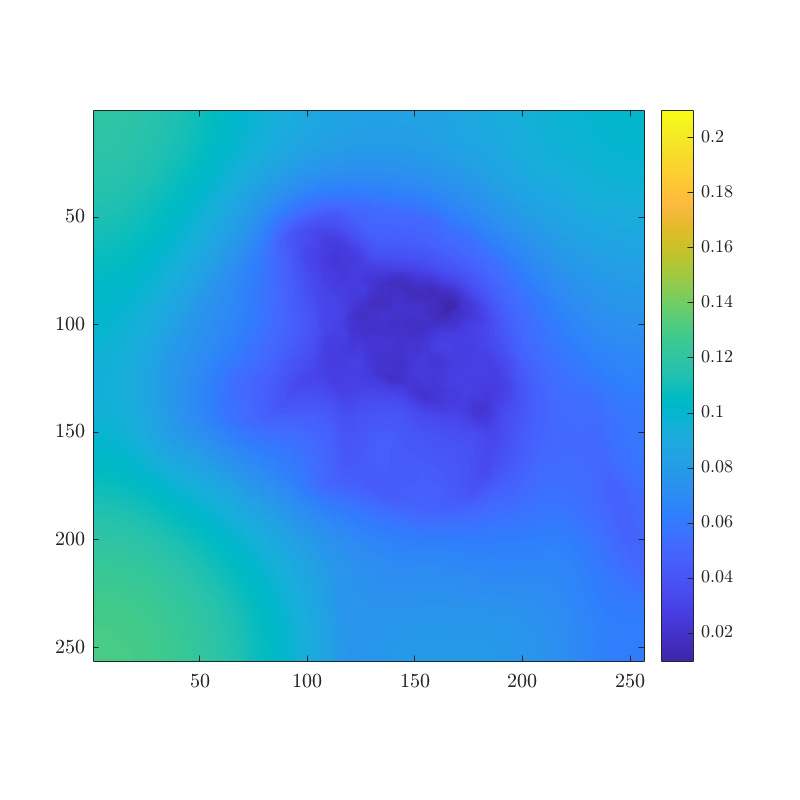}
	\end{minipage}
	  \begin{minipage}[t]{0.23\textwidth}
	\includegraphics[width=0.95\textwidth, trim={0cm 3.5cm 0cm 0cm},clip]{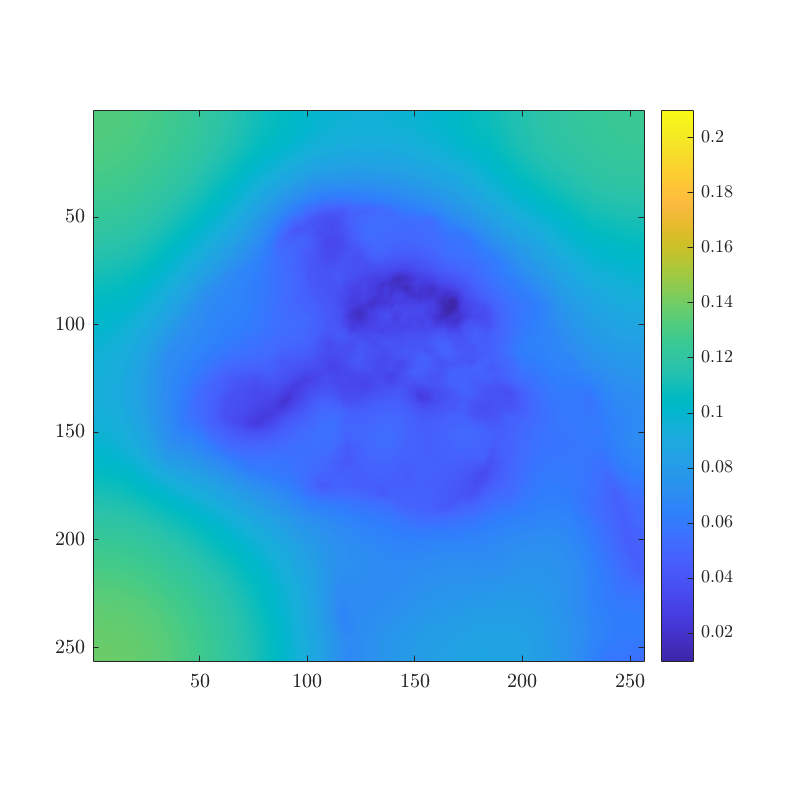}
	\end{minipage}

	\begin{minipage}[t]{0.23\textwidth}
	\centering \scalebox{.65}{Weight $\alpha$ of Huber TV}
	\end{minipage}
	\begin{minipage}[t]{0.23\textwidth}
	\centering \scalebox{.65}{Weight $\alpha$ of Huber TV}
	\end{minipage}	
	\begin{minipage}[t]{0.23\textwidth}
\centering \scalebox{.65}{Weight $\alpha$ of Huber TV$^{2}$}
	\end{minipage}
	\begin{minipage}[t]{0.23\textwidth}
\centering \scalebox{.65}{Weight $\alpha$ of Huber TV$^{2}$}
	\end{minipage}
	
	\vspace{1em}
	\begin{minipage}[t]{0.23\textwidth}
	\includegraphics[width=0.95\textwidth, trim={0cm 0cm 0cm 0cm},clip]{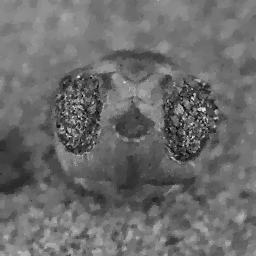}
	\end{minipage}
	\begin{minipage}[t]{0.23\textwidth}
	\includegraphics[width=0.95\textwidth, trim={0cm 0cm 0cm 0cm},clip]{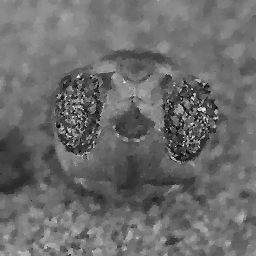}
	\end{minipage}
	\begin{minipage}[t]{0.23\textwidth}
	\includegraphics[width=0.95\textwidth, trim={0cm 0cm 0cm 0cm},clip]{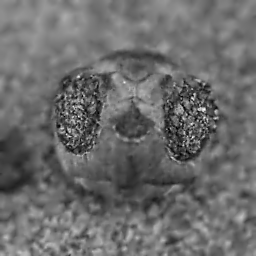}
	\end{minipage}
	\begin{minipage}[t]{0.23\textwidth}
	\includegraphics[width=0.95\textwidth, trim={0cm 0cm 0cm 0cm},clip]{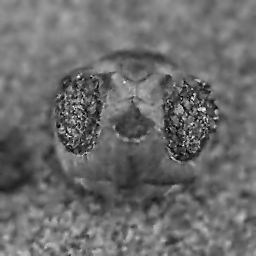}
	\end{minipage}
	
	\begin{minipage}[t]{0.23\textwidth}
	 \centering \scalebox{.65}{Bilevel weighted Huber TV}  \\[-3pt] \scalebox{.65}{with scalar  $\gamma$}  \\\scalebox{.65}{PSNR=28.23, SSIM=0.7979}
	\end{minipage}
	\begin{minipage}[t]{0.23\textwidth}
	 \centering \scalebox{.65}{Bilevel weighted  Huber TV}  \\[-3pt] \scalebox{.65}{with spatially varying  $\gamma$}  \\\scalebox{.65}{PSNR=27.86, SSIM=0.8016}
	\end{minipage}
	\begin{minipage}[t]{0.23\textwidth}
	 \centering \scalebox{.65}{Bilevel weighted  Huber TV$^{2}$}  \\[-3pt] \scalebox{.65}{with scalar  $\gamma$}  \\ \scalebox{.65}{PSNR=\textbf{29.09}, SSIM=\textbf{0.8494}}
	\end{minipage}
	\begin{minipage}[t]{0.23\textwidth}
	 \centering \scalebox{.65}{Bilevel weighted  Huber TV$^{2}$}  \\[-3pt] \scalebox{.65}{with spatially varying  $\gamma$}  \\ \scalebox{.65}{PSNR=28.96, SSIM=0.8492}
	\end{minipage}
		
	\begin{minipage}[t]{0.23\textwidth}
	\includegraphics[width=0.95\textwidth, trim={0cm 3.5cm 0cm 0cm},clip]{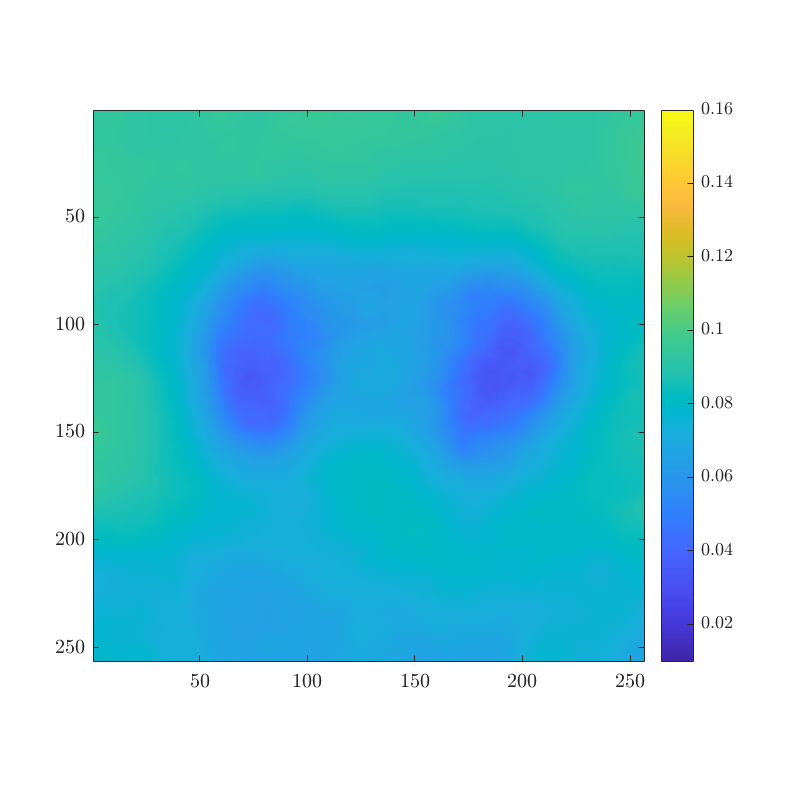}
	\end{minipage}	
	\begin{minipage}[t]{0.23\textwidth}
	\includegraphics[width=0.95\textwidth, trim={0cm 3.5cm 0cm 0cm},clip]{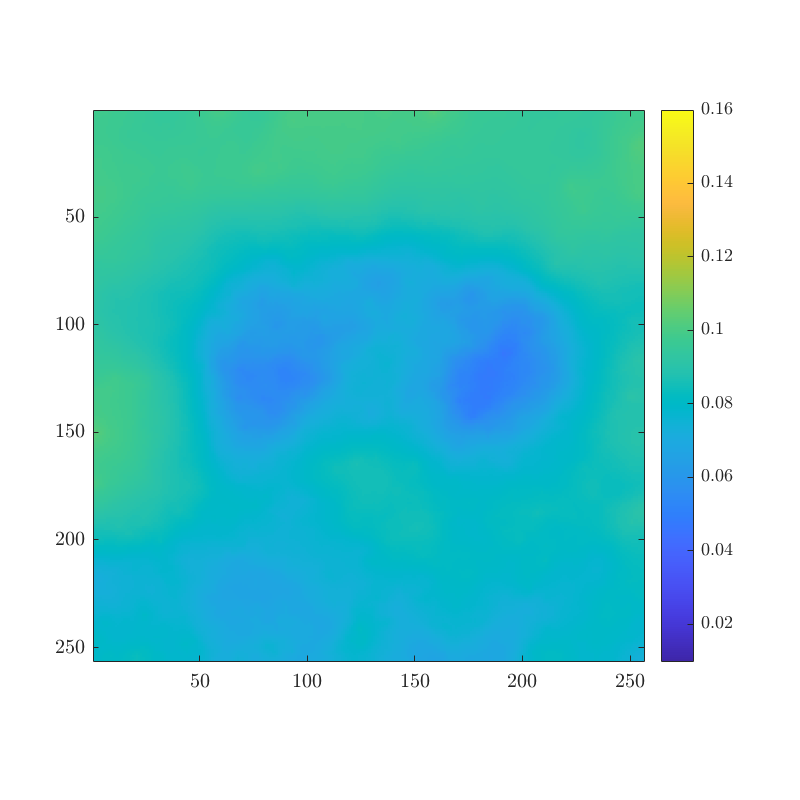}
	\end{minipage}	
	  \begin{minipage}[t]{0.23\textwidth}
	\includegraphics[width=0.95\textwidth, trim={0cm 3.5cm 0cm 0cm},clip]{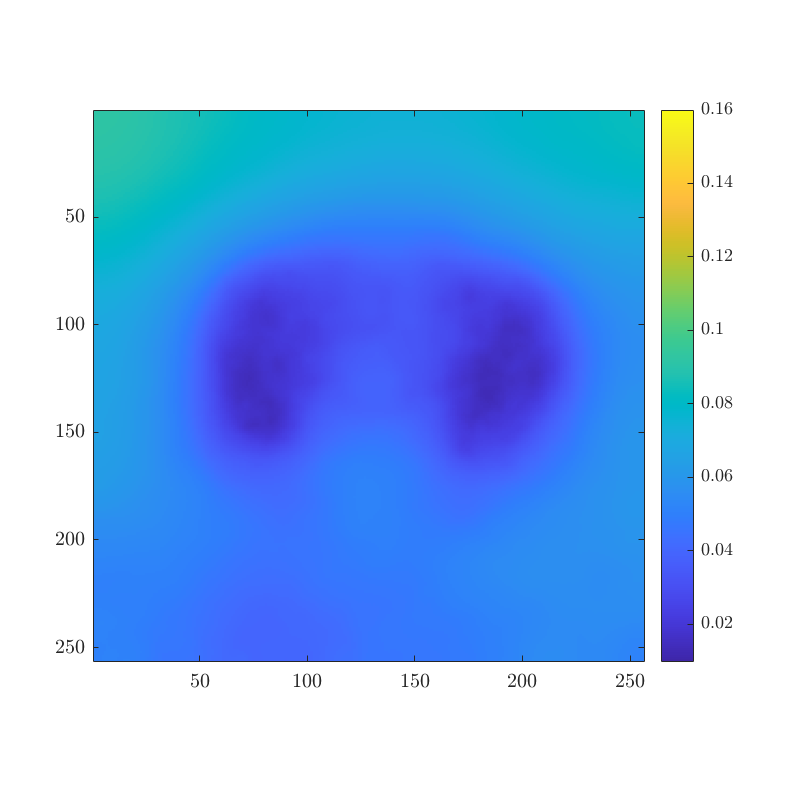}
	\end{minipage}
	\begin{minipage}[t]{0.23\textwidth}
	\includegraphics[width=0.95\textwidth, trim={0cm 3.5cm 0cm 0cm},clip]{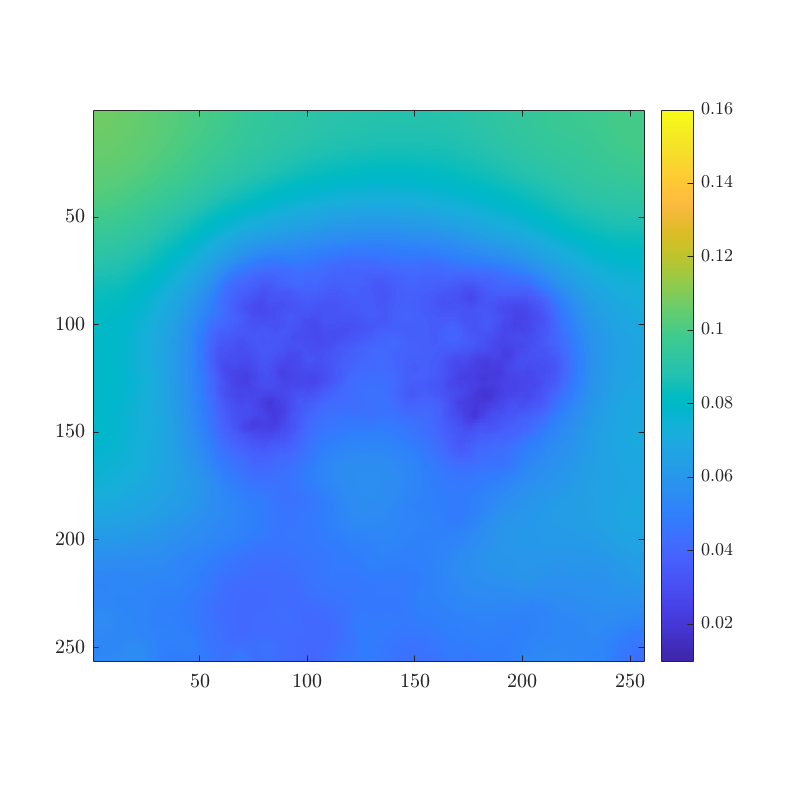}
	\end{minipage}
	
	\begin{minipage}[t]{0.23\textwidth}
	\centering \scalebox{.65}{Weight $\alpha$ of Huber TV}
	\end{minipage}
	\begin{minipage}[t]{0.23\textwidth}
	\centering \scalebox{.65}{Weight $\alpha$ of Huber TV}
	\end{minipage}
	\begin{minipage}[t]{0.23\textwidth}
\centering \scalebox{.65}{Weight $\alpha$ of Huber TV$^{2}$}
	\end{minipage}
	\begin{minipage}[t]{0.23\textwidth}
\centering \scalebox{.65}{Weight $\alpha$ of Huber TV$^{2}$}
	\end{minipage}	
	\caption{Weighted Huber TV and TV$^{2}$ denoising with spatially varying Huber parameter $\gamma$ and regularization parameter $\alpha$. The weights $\alpha$ are produced with the ground truth-based bilevel approach using $\mathcal{F}_{\PSNR}$. The highest PSNR and SSIM values are highlighted in bold font. }	
			\label{fig:psnr_opt}
\end{figure}

\begin{figure}[h!]
	\centering		
	\begin{minipage}[t]{0.48\textwidth}
	\includegraphics[width=0.97\textwidth, trim={0cm 0cm 0cm 0cm},clip]{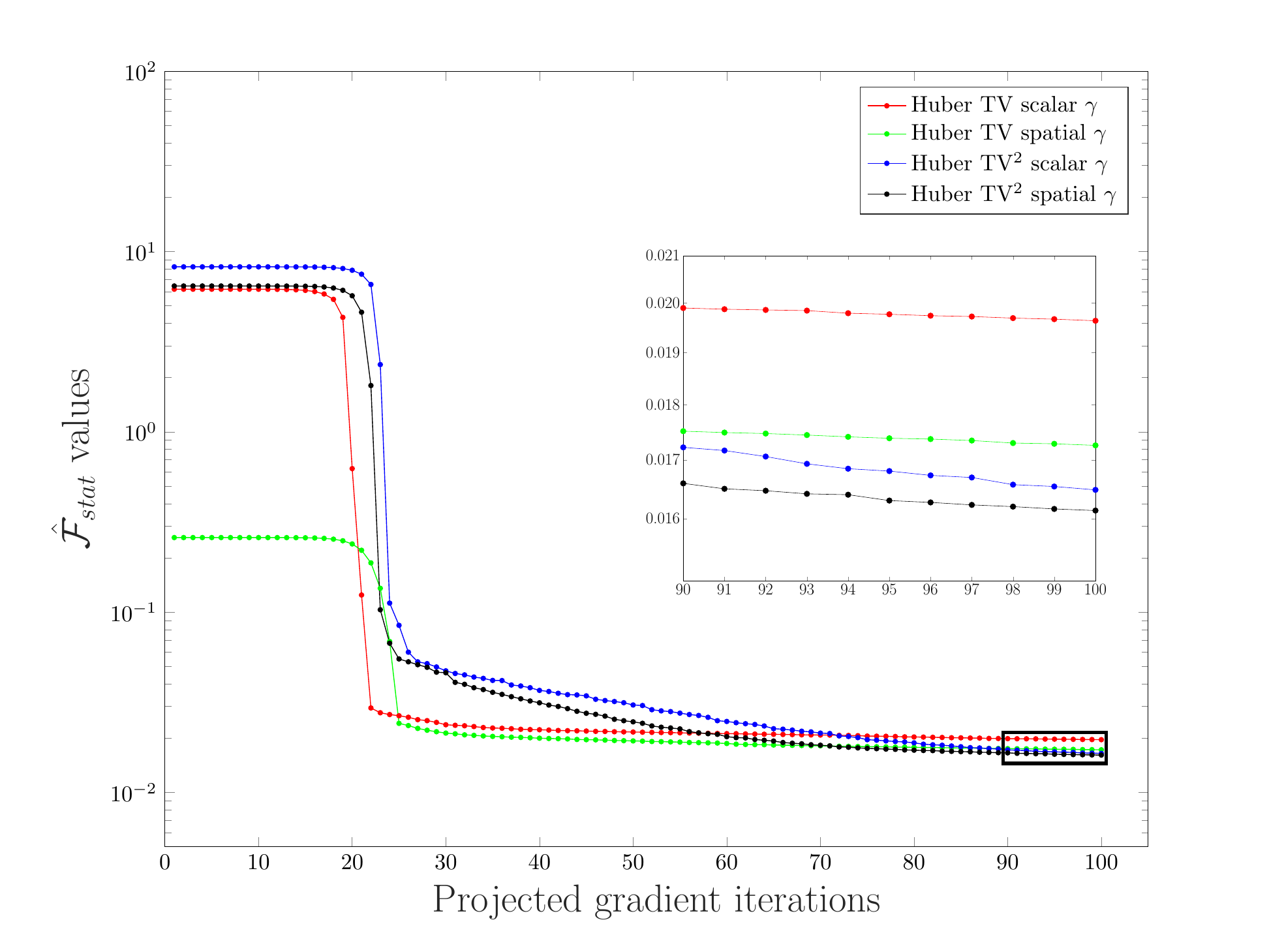}
	\end{minipage}
	\begin{minipage}[t]{0.48\textwidth}
	\includegraphics[width=0.97\textwidth, trim={0cm 0cm 0cm 0cm},clip]{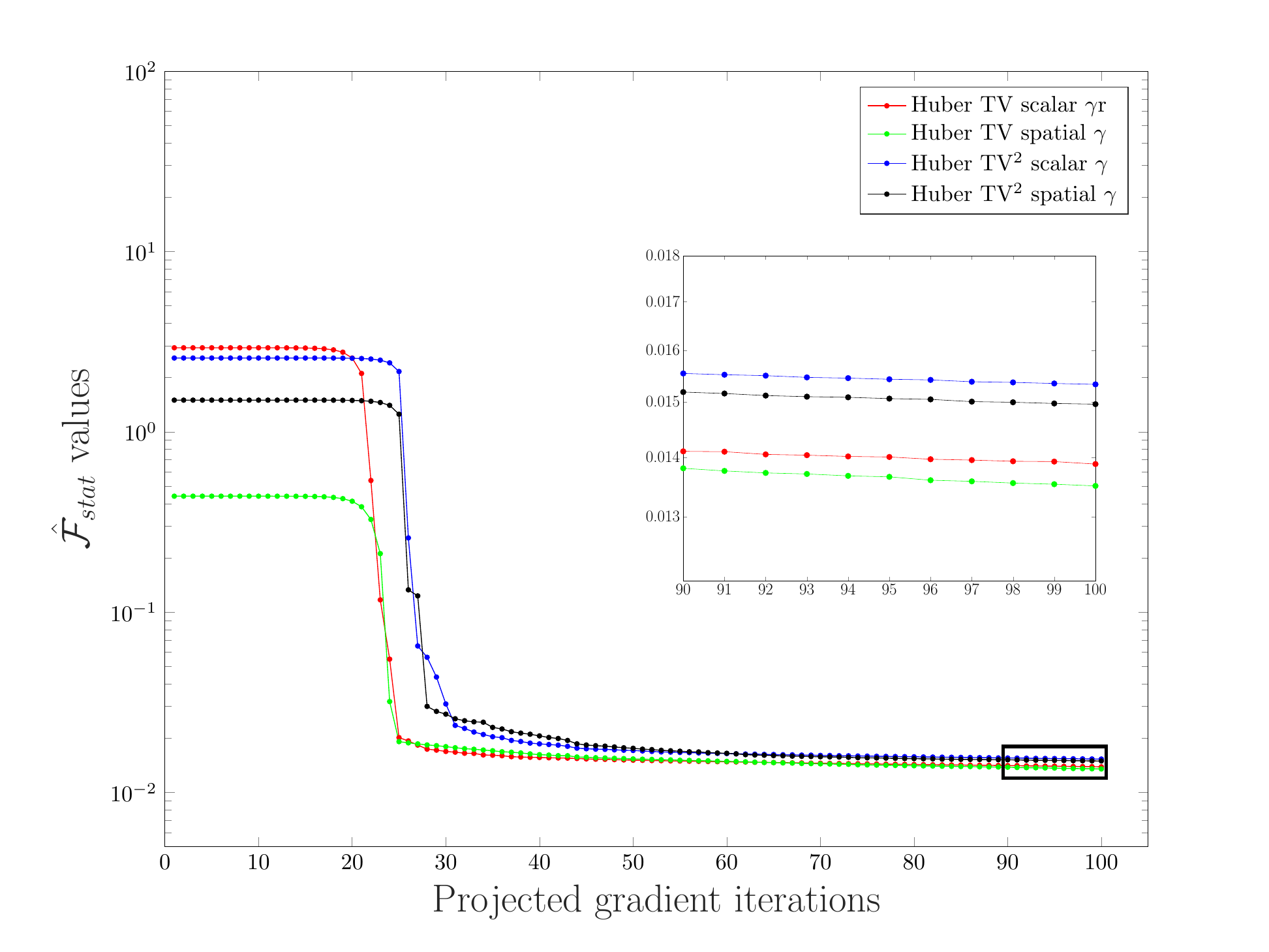}
	\end{minipage}\vspace{-0.5em}
	
	\begin{minipage}[t]{0.48\textwidth}
	\centering \scalebox{.65}{$\hat{\mathcal{F}}_{\stat}$ for \emph{parrot}}
	\end{minipage}
	\begin{minipage}[t]{0.48\textwidth}
	\centering \scalebox{.65}{$\hat{\mathcal{F}}_{\stat}$ for \emph{hatchling}}
	\end{minipage}
	
	\begin{minipage}[t]{0.48\textwidth}
	\includegraphics[width=0.97\textwidth, trim={0cm 0cm 0cm 0cm},clip]{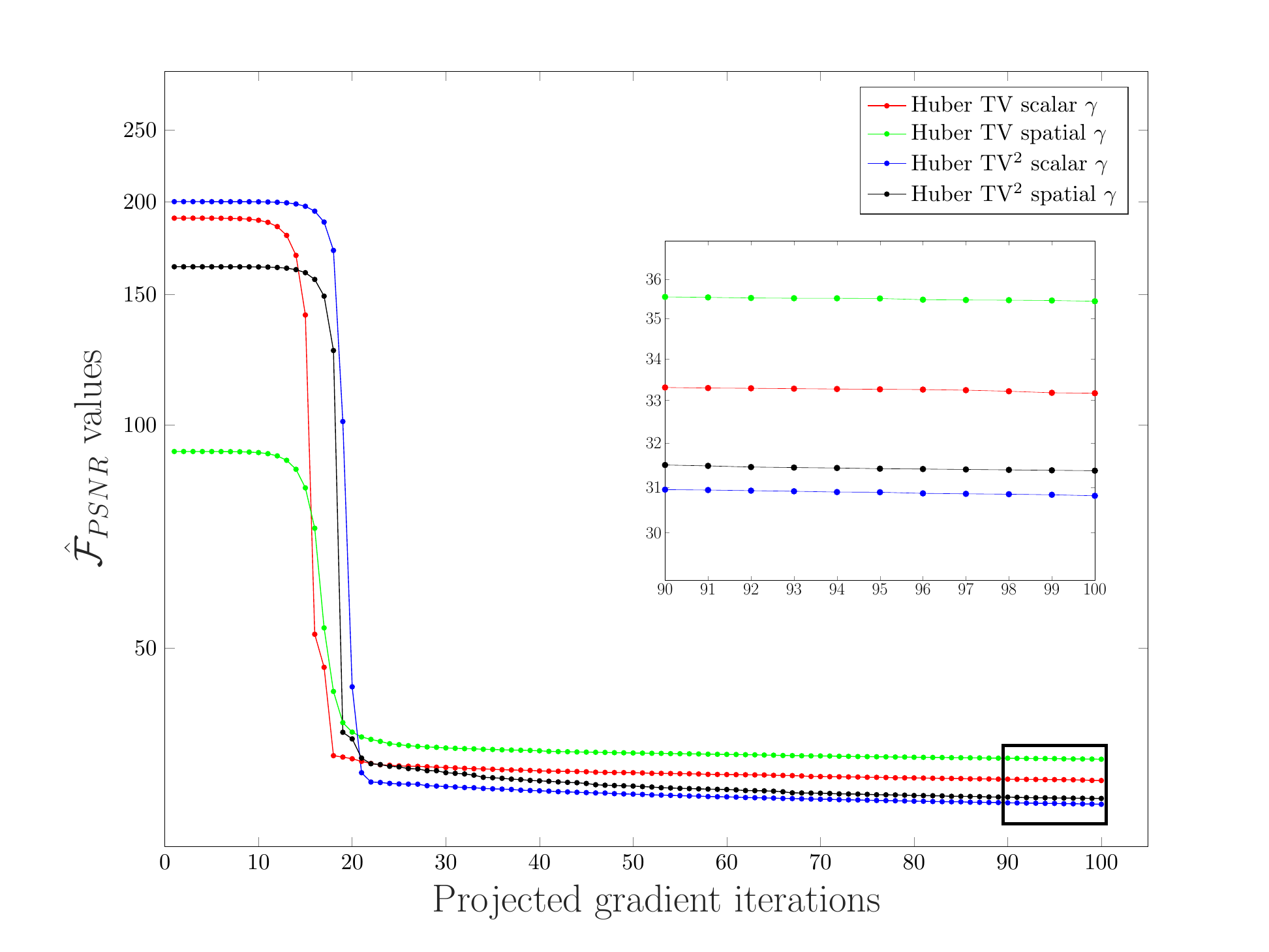}
	\end{minipage}
	\begin{minipage}[t]{0.48\textwidth}
	\includegraphics[width=0.97\textwidth, trim={0cm 0cm 0cm 0cm},clip]{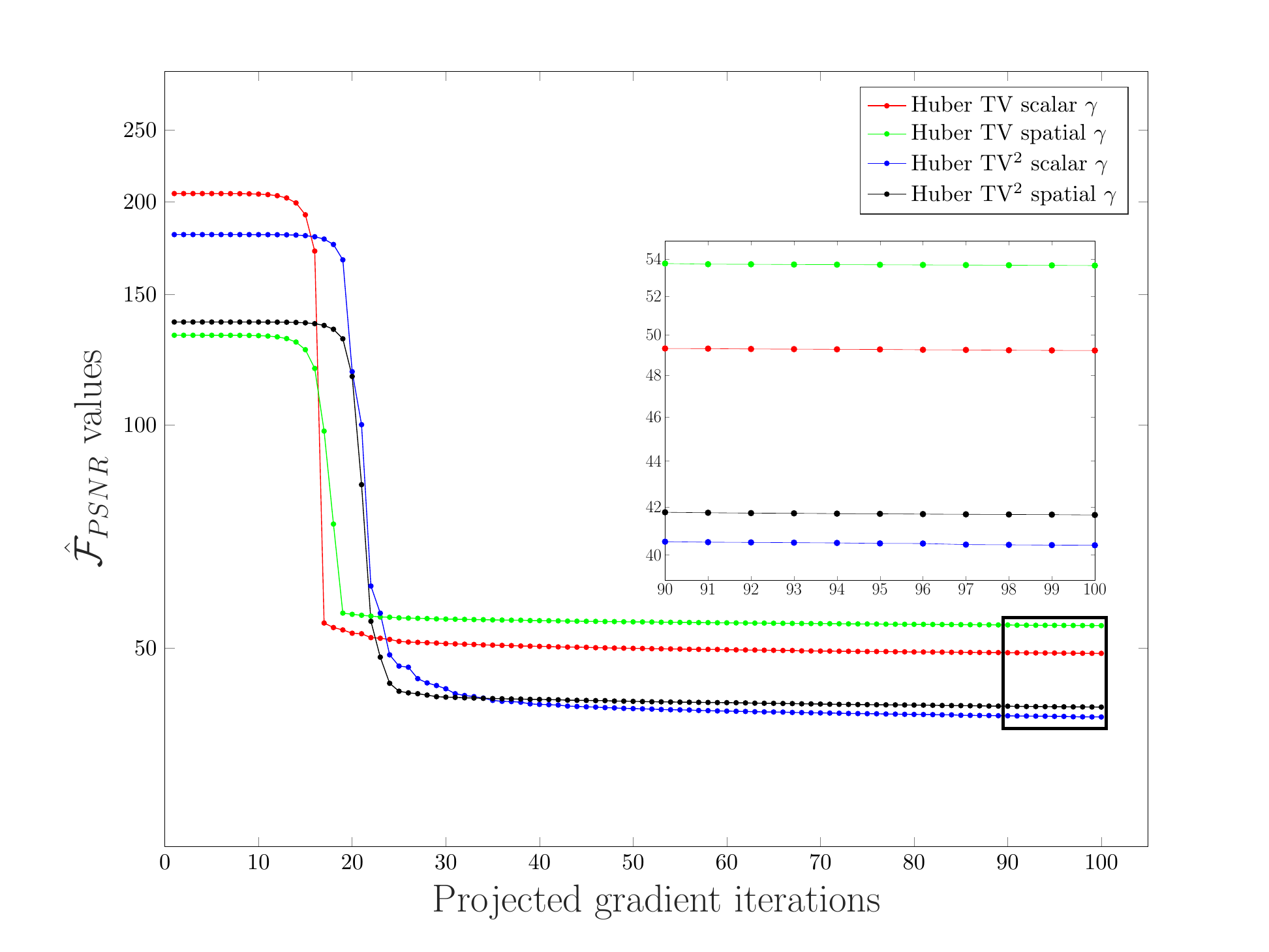}
	\end{minipage}\vspace{-0.5em}
	
	\begin{minipage}[t]{0.48\textwidth}
	\centering \scalebox{.65}{$\hat{\mathcal{F}}_{\PSNR}$ for \emph{parrot}}
	\end{minipage}
	\begin{minipage}[t]{0.48\textwidth}
	\centering \scalebox{.65}{$\hat{\mathcal{F}}_{\PSNR}$ for \emph{hatchling}}
	\end{minipage}
	\caption{Values of the reduced objective $\hat{\mathcal{F}}(u^{k})$ along the projected gradient iterations. The inner boxes show zoom-in plots of the last 10 iterations.}
		\label{fig:Fvalues}
\end{figure}	
	
\subsection{Numerical results}
In Figure \ref{fig:parrot_stat} we report our numerical results on the \emph{Parrot} image, see also Figure \ref{fig:parrot_stat_details} for zoom-in details. Here the spatially varying  regularization weights $\alpha$ are produced with the ground truth-free bilevel approach, i.e., using $\mathcal{F}_{\stat}$ as an upper lever objective. Among the regularizers with scalar parameters, second row, first three images, the best reconstruction both in terms of PSNR and SSIM is achieved by the scalar TGV. The bilevel weighted Huber TV reconstruction with scalar $\gamma$ is able to better preserve the details around the eye of the parrot, third row first image. When we use the spatially varying  $\gamma$, the details in that area become even more pronounced, compare the first two images in the third row of Figure \ref{fig:parrot_stat_details}. This is also accompanied with a slight increase of the SSIM index but also with a decrease in  PSNR. Observe that the weights $\alpha$ that are computed in these two cases are quite different, see first two images of the last row of Figure \ref{fig:parrot_stat}. The bilevel weighted Huber TV$^{2}$ approach produces similar reconstructions for both the scalar (slighly higher SSIM) and the spatially varying $\gamma$ case (slightly higher PSNR). These reconstructions are of very good quality and even outperform the weighted TGV in terms of SSIM, having also the same PSNR. This is due to the fact that the combination of the statistics-based upper level objective  and the second order TV is forcing the weight $\alpha$ to drop significantly in the detailed areas of the image, see the last two images of the last row of Figure \ref{fig:parrot_stat}. It is characteristic that while the PSNR of scalar TV$^{2}$ is only $0.03$ $\mathrm{dB}$ higher than the one of scalar TV, the PSNR of bilevel weighted huber TV$^{2}$ with scalar $\gamma$ is 0.61 $\mathrm{dB}$ higher compared to the corresponding huber TV result.	
	
The superiority of the bilevel weighted Huber TV$^{2}$ is even more evident in the second image example \emph{Hatchling}, Figures \ref{fig:hatchling_stat} and \ref{fig:hatchling_stat_details}. Here the reconstruction is more challenging due to the oscillatory nature of the ground truth image. The bilevel Huber TV$^{2}$ with scalar $\gamma$ gives by far the best result with respect to both PSNR and SSIM. Again, the automatically computed regularization weights $\alpha$ have much lower values in bilevel TV$^{2}$ than in bilevel TV, compare the first two versus the last two figures of the last row of Figure \ref{fig:hatchling_stat}. In this example, the spatially varying $\gamma$ leads to a reduction of PSNR and SSIM in all cases, but nevertheless also to more highlighted details in the eye area, see second and fourth images of the last row of Figure \ref{fig:hatchling_stat_details}.

In order to verify further the regularization capabilities of these regularizers, we make another series of experiments with these two example images, using the ground truth-based upper level objective $\mathcal{F}_{\PSNR}$, see Figure \ref{fig:psnr_opt}. In both images, the highest PSNR and SSIM is achieved by the bilevel weighted TV$^{2}$ with scalar Huber parameter $\gamma$, third images of first and third row, with the corresponding the regularization weight $\alpha$ having again smaller values compared to the TV one. Nevertheless, we observe that the spatially varying $\gamma$ results in higher SSIM in the Huber TV examples in both images, again with more pronounced features around the eye.

Finally in Figure \ref{fig:Fvalues}, we have plotted the values of the reduced objective $\hat{\mathcal{F}}(u^{k})$ along the projected gradient iterations, for  all bilevel Huber TV and  TV$^{2}$ examples. The top row shows these plots for the reduced  statistics-based upper level objective $\hat{\mathcal{F}}_{\stat}$. We observe that in both images, the introduction of the spatially varying $\gamma$ in both Huber TV and Huber TV$^{2}$ functionals, helps towards a further reduction of this objective, compare red versus green and blue versus black plots. We observed already that in some cases this is accompanied with a larger SSIM index and more pronounced details in the images, but in most cases the PSNR in decreased. This is in accordance with the plots of the second row, where we see that  the reduced  PSNR-maximizing upper level objective $\hat{\mathcal{F}}_{\PSNR}$ is not further decreased by the introduction of the spatially varying $\gamma$, compare again  the red versus green and blue versus black plots.

We conclude that the bilevel Huber TV$^{2}$ is able to produce remarkably good results. This is perhaps even surprising as the use of its scalar version is not that popular due to its inability to preserve sharp edges.  
We showed that the use of a spatially varying Huber parameter $\gamma$ can result in improved results both quantitatively and qualitatively, thus justifying our rigorous analytical study
on spatially inhomogeneous integrands acting on TV-type regularizers. We also stress that by no means our strategy for setting $\gamma$ is necessarily the optimal one. In fact, future work will involve setting up a bilevel framework where also this parameter is included in the upper level minimization variables along with the  parameter $\alpha$, adding further flexibility to the regularization process.

\appendix 
\section{Convex integrands and generalized Young measures}\label{sec:YM}

This section loosely follows \cite{KrRa_notes}, where most proofs can be found. Let $\Omega\subset\R^n$ be a bounded open set with $\mathscr{L}^n(\partial\Omega)=0$ and consider the space of integrands
$$
\mathbb{E}(\Omega,\mathbb{V})\coloneqq \left\{\Phi\in\hold(\Omega\times\mathbb{V})\colon\Phi^\infty(x,z)\coloneqq \lim_{t\rightarrow\infty,\,x^\prime\rightarrow x}\frac{\Phi(x^\prime,tz)}{t}\in \R\text{ uniformly in }\bar\Omega\times S_\V \right\},
$$
which is naturally equipped with the norm
$$
\|\Phi\|_{\mathbb{E}}\coloneqq \sup_{(x,z)\in\Omega\times\V}\dfrac{|\Phi(x,z)|}{1+|z|}.
$$
It will thus be convenient to work with the  coordinate transformations
$$
S\colon\hat{z}\in B_\V\mapsto \frac{\hat{z}}{1-|\hat z|}\in\mathbb{V},
\quad 
S^{-1}\colon z\in\mathbb{V}\mapsto \frac{{z}}{1+| z|}\in B_\V,
$$
where we wrote $B_\V$ to denote the open unit ball in $\mathbb{V}$. With this notation, the space of integrands $\mathbb{E}(\Omega,\mathbb{V})$ can be identified with $\hold(\overline{\Omega\times B_\V})$ via the linear isometric isomorphism
$$
(T\Phi)(x,\hat z)\coloneqq(1-|\hat z|)\Phi \left(x,\frac{\hat z}{1-|\hat z|}\right), \quad\text{for }x\in\Omega,\,\hat z\in B_\V.
$$
It follows that its adjoint, $T^*\colon\mathbb E(\Omega,\mathbb{V})^*\rightarrow \hold(\overline{\Omega\times B_\V})^*\cong \mathcal{M}(\overline{\Omega\times B_\V})$ is also a linear isometric isomorphism. We embed $\mathcal M(\Omega,\mathbb V)$ into $\mathbb{E}^*$ via
\begin{align*}
\eps_\mu(\Phi)\coloneqq\int_{\Omega}\Phi(x,\dif \mu)=\int_\Omega\Phi(x,\mu^a(x))\dif x+\int_\Omega\Phi^\infty\left(x,\frac{\dif\mu^s}{\dif \mu}(x)\right)\dif|\mu|(x),
\end{align*}
where $\mu=\mu^a\mathscr{L}^n\mres\Omega+\mu^s$ is the Radon--N\'ykodim decomposition of $\mu\in\mathcal M(\Omega,\V)$. By the sequential Banach--Alaoglu theorem, we can infer that bounded $L^p$ sequences are weakly-* compact in $\mathbb{E}^*$ under the above identification. In particular, if $(\mu_j)$ is bounded in $\mathcal M(\Omega,\V)$, we know that along a subsequence we have $\eps_{\mu_j}\wstar \bm\nu$ in $\mathbb{E}(\Omega,\mathbb{V})^*$. We define $\sigma\coloneqq(T^{-1})^*\bm\nu\in\mathcal{M}(\overline{\Omega\times B_\V})$ and write for $\Phi\in \mathbb{E}(\Omega,\V)$
\begin{align*}
\llangle \Phi,\bm\nu\rrangle&\coloneqq\langle\Phi,\bm\nu\rangle_{\mathbb{E},\mathbb{E}^*}=\langle T\Phi, \sigma \rangle\\
&=\int_{\bar\Omega\times B_\V }(1-|\hat z|)^p\Phi \left(x,\frac{\hat z}{1-|\hat z|}\right)\dif\sigma(x.\hat z)+\int_{\bar\Omega\times S_\V}\Phi^\infty(x,\hat z)\dif \sigma(x,\hat z).
\end{align*}
From this formula we derive two necessary conditions for the weakly-* limits of $\eps_{\mu_j}$, namely that $\sigma\geq 0 \text{ in the sense of }\mathcal{M}(\overline{\Omega\times B_\V})$ and
\begin{align}\label{eq:nec_cond}
\int_{\Omega}\varphi(x)\dif x=\int_{\bar\Omega\times B_\V}\varphi(x)(1-|\hat z|)\dif \sigma(x,\hat z)\text{ for all }\varphi\in\hold(\bar\Omega).
\end{align}
Conversely, these conditions are sufficient to enable us to disintegrate $\sigma$ into appropriately parametrized (generalized Young) measures that detect both oscillation and concentration behavior of a weakly-* convergent sequence $(\mu_j)$. We define:
\begin{definition}
	A parametrized measure $\bm\nu=\left((\nu_x)_{x\in\Omega},\,\lambda,\,(\nu_x^\infty)_{x\in\bar\Omega}
	\right)$ is said to be a \emph{Young measure} (or \emph{generalized Young measure}) whenever 
	\begin{enumerate}
		\item $(\nu_x)_{x\in\Omega}\subset\mathcal{M}^+_1(\mathbb{V})$ is weakly-* $\mathscr{L}^n$-measurable (the \emph{oscillation measure}).
		\item $\lambda\in\mathcal{M}^+(\bar\Omega)$ (the \emph{concentration measure}).
		\item $(\nu_x^\infty)_{x\in\bar\Omega}\subset\mathcal{M}^+_1(\mathbb{V})$ is weakly-* $\lambda$-measurable (the \emph{concentration-angle measure}).
		\item $\int_{\Omega}\int_{\mathbb{V}}|z|\dif \nu_x(z)\dif x<\infty$ (the \emph{moment condition} holds).
	\end{enumerate}
	Then $\bm\nu$ acts linearly on $\mathbb{E}(\Omega,\mathbb{V})$ via
	$$
	\llangle \Phi, \bm\nu \rrangle\coloneqq\int_{\Omega }\int_{\mathbb{V}} \Phi(x,\,\bigcdot\,)\dif\nu_x\dif x+\int_{\bar\Omega}\int_{ S_\V} \Phi^\infty(x,\,\bigcdot\,)\dif\nu_x^\infty\dif\lambda(x) \quad\text{for } \Phi\in\mathbb{E}(\Omega,\mathbb V).
	$$
	We write $\mathrm{Y}(\Omega,\mathbb{V})$  for the set of all such $\bm\nu$.
\end{definition}
It is then easy to check that a Young measure $\bm\nu$ actually lies in $\mathbb{E}^*$ and, moreover, that the inclusion $\mathrm{Y}(\Omega,\mathbb{V})\subset\mathbb{E}(\Omega,\mathbb{V})^*$ is strict. We have the disintegration theorem:
\begin{theorem}\label{thm:disintegration}	$\mathrm{Y}(\Omega,\mathbb{V})=T^*\{\sigma\in\mathcal{M}^+(\overline{\Omega\times B_\V})\colon\text{ equation \eqref{eq:nec_cond} holds} \}$.
\end{theorem}

Consequently, $\mathrm{Y}(\Omega,\mathbb{V})$ is weakly-* closed in $\mathbb{E}(\Omega,\mathbb{V})^*$ and convex, therefore:
\begin{theorem}[Fundamental Theorem of Young measures]\label{thm:ftym}
	Let $(\mu_j)$ be a bounded sequence in $\mathcal M(\Omega,\mathbb{V})$. Then there exists $\bm\nu\in\mathrm{Y}(\Omega,\mathbb V)$ such that, along a subsequence, $\eps_{\mu_j}\wstar \bm\nu$ in $\mathbb{E}(\Omega,\mathbb{V})^*$, i.e.,
	$$
	\lim_{j\rightarrow\infty}\int_{\Omega}\Phi(x,\dif \mu_j)=\int_{\Omega }\int_{\mathbb{V}} \Phi(x,z)\dif\nu_x(z)\dif x+\int_{\bar\Omega}\int_{ S_\V} \Phi^\infty(x,z)\dif\nu_x^\infty(z)\dif\lambda(x)
	$$
	for all $\Phi\in\mathbb{E}(\Omega,\mathbb V)$. In this case, we say that \emph{$(\mu_j)$ generates $\bm\nu$}.
\end{theorem}
In our analysis we repeatedly use the following:
\begin{lemma}\label{lem:barycentre}
Let $(\mu_j)\subset \mathcal M(\Omega,\mathbb{V})$ generate a Young measure $\bm\nu$. Then
$$
\mu_j\wstar \bar\nu_x\mathscr{L}^n\mres\Omega +\bar\nu_x^\infty\lambda \text{ in }\mathcal{M}(\bar\Omega.\V).
$$
The limit measure is refered to as the \emph{barycentre} of $\bm\nu$.
\end{lemma}

This follows simply by taking $\Phi(x,z)=\varphi(x)z_i$ for $\varphi\in\hold(\bar{\Omega})$,    where we wrote $$\bar\nu_x=\int_\V z\dif\nu_x(z)\quad\text{and}\quad\bar\nu_x^\infty=\int_{B_\V} z\dif\nu^\infty_x(z)$$
for the expectations of the probability measures $\nu_x$ and $\nu_x^\infty$.

We also employ a general convergence result for Young measures:
\begin{proposition}{\cite[Prop. 2(i)]{KrRi}}\label{prop:lsc_YM}
Let $\Omega\subset\R^n$ be bounded and open, and $F\colon \Omega\times \V\rightarrow\R$ be a measurable integrand such that $f(x,\,\bigcdot\,)$ is continuous for almost every $x\in\Omega$ (a Carath\'eodory integrand). Suppose in addition that $f$ has a regular recession function, i.e., 
$$
f^\infty(x,z)\coloneqq \lim_{(x',z',t)\rightarrow(x,z,+\infty)}\frac{f(x',tz')}{t}\quad\text{for }(x,z)\in\bar\Omega\times\V
$$
exists. Let $(\mu_j)\subset\mathcal M(\Omega,\V)$ generate $\ym\in\Y(\Omega,\V)$. Then
$$
\lim_{j\rightarrow\infty} \int_{\Omega}f(x,\dif \mu_j(x))= \int_{\Omega}\langle\nu_x, f(x,\,\bigcdot\,)\rangle\dif x+\int_{\bar\Omega}\langle\nu_x^\infty, f^\infty(x,\,\bigcdot\,)\rangle\dif \lambda(x).
$$
\end{proposition}
Finally, we cite a variant of the fundamental theorem of Young measures:
\begin{proposition}\label{prop:cts_YM}
Let $\Omega\subset\R^n$ be bounded and open, and $F\colon \Omega\times \V\rightarrow\R$ be a measurable integrand such that $f(x,\,\bigcdot\,)$ is continuous for almost every $x\in\Omega$ (a Carath\'eodory integrand). Let $(v_j)\subset L^1(\Omega,\V)$ generate $\ym\in\Y(\Omega,\V)$ be such that $(f(\,\bigcdot\,,v_j))_j$ is uniformly integrable. Then
$$
\lim_{j\rightarrow\infty} \int_{\Omega}f(x, v_j(x))\dif x= \int_{\Omega}\langle\nu_x, f(x,\,\bigcdot\,)\rangle\dif x.
$$
\end{proposition}

\bibliographystyle{amsplain}
\bibliography{kostasbib}

\end{document}